\documentclass[letterpaper,11pt]{article}
\usepackage{color}
\usepackage{pdflscape}

\usepackage[small]{titlesec}
\usepackage{theorem,amsmath,amssymb,amscd}
\usepackage[all,cmtip]{xy}
\usepackage{booktabs}
\usepackage[hyperfootnotes=false, colorlinks, linkcolor={blue}, citecolor={magenta}, filecolor={blue}, urlcolor={blue}]{hyperref}
\usepackage[overload]{textcase} % Suppresses automatic capitalization of math code in page headers

\theoremstyle{change}
\allowdisplaybreaks
\newcommand{\Q}{{\mathbb Q}}
\newcommand{\Z}{{\mathbb Z}}

\newcommand{\C}{{\mathbb C}}

\newcommand{\p}{\mathfrak p}
\newcommand{\OF}{{\mathfrak o}}
\newcommand{\Fq}{{\mathbb F}_q}
\newcommand{\GL}{{\rm GL}}

\newcommand{\SL}{{\rm SL}}

\newcommand{\GSp}{{\rm GSp}}

\newcommand{\Sp}{{\rm Sp}}

\newcommand{\K}[1]{{\rm K}(\p^{#1})}

\newcommand{\Si}[1]{{\rm Si}(\p^{#1})}

\newcommand{\val}{{\rm val}}

\newcommand{\cInd}{\text{\rm c-Ind}}
\newcommand{\trace}{{\rm tr\,}}

\newcommand{\Supp}{{\rm Supp}}

\newcommand{\qed}{\hspace*{\fill}\rule{1ex}{1ex}}
\newcommand{\forget}[1]{}
\def\qdots{\mathinner{\mkern1mu\raise0pt\vbox{\kern7pt\hbox{.}}\mkern2mu
\raise3.4pt\hbox{.}\mkern2mu\raise7pt\hbox{.}\mkern1mu}}

\newenvironment{proof}{\vspace{1ex}\noindent\emph{Proof.}\hspace{0.5em}}
	{\hfill\qed\vspace{2ex}}
\newenvironment{bsmallmatrix}{\left[\begin{smallmatrix}}{\end{smallmatrix}\right]}

\newtheorem{lemma}{Lemma.}[section]
\newtheorem{theorem}[lemma]{Theorem.}
\newtheorem{corollary}[lemma]{Corollary.}

\newtheorem{remark}[lemma]{Remark.}

\begin{document}

\thispagestyle{empty}

\begin{center}
 {\bf\Large Siegel Vectors for Nongeneric Depth Zero Supercuspidals of $\GSp(4)$}

 \vspace{3ex}
 Jonathan Cohen
 
 \vspace{3ex}
 \begin{minipage}{80ex}
  \small\textbf{Abstract.} Let $F$ be a non-archimedean local field of characteristic zero. If $F$ has even residual characteristic, we assume $F/\Q_2$ is unramified. Let $V$ be a depth zero, irreducible, nongeneric supercuspidal representation of $\GSp(4, F)$. We calculate the dimensions of the spaces of Siegel-invariant vectors in $V$ of level $\p^n$ for all $n\geq0$. 
 \end{minipage}
 \vspace{3ex}
\end{center}

\section{Introduction}

This paper is concerned with a dimension counting problem in $p$-adic representation theory for the group $\GSp(4)$. Part of the motivation comes from the following problem in the classical theory of Siegel modular forms. If $\Gamma\subset \Sp(4, \Q)$ is a congruence subgroup, then the dimension of the space of cusp forms of level $\Gamma$ (and integer weight $k$, say) is not known in general. For example, if $$\Gamma=\begin{bmatrix}
\Z & \Z & \Z & \Z \\
\Z & \Z & \Z & \Z \\
p^n\Z & p^n\Z & \Z & \Z \\
p^n\Z & p^n\Z & \Z & \Z 
\end{bmatrix}\cap \Sp(4, \Z)$$ is the ``Siegel congruence subgroup of level $p^n$'' then the associated dimensions are unknown if $n\geq 2$, for any prime $p$. On the other hand, if $$\Gamma=\begin{bmatrix}
\Z & 4\Z & \Z & \Z \\
\Z & \Z & \Z & \Z \\
\Z & 4\Z & \Z & \Z \\
4\Z & 4\Z & 4\Z & \Z 
\end{bmatrix}\cap \Sp(4, \Z)$$ is the ``Klingen congruence subgroup of level $4$'' then the associated dimensions were recently computed in \cite{RSY2022}. The method by which this was achieved required, as one of its several inputs, the  dimensions of spaces of fixed vectors in all irreducible smooth representations of $\GSp(4, \Q_2)$ for the subgroups $${\rm Kl}(4)=\begin{bmatrix}
\Z_2 & \Z_2 & \Z_2 & \Z_2 \\
4\Z_2 & \Z_2 & \Z_2 & \Z_2 \\
4\Z_2 & \Z_2 & \Z_2 & \Z_2 \\
4\Z_2 & 4\Z_2 & 4\Z_2 & \Z 
\end{bmatrix}\cap \GSp(4, \Z_2).$$ These dimensions, and more, had been computed in \cite{Yi2021}. We remark that the different ``shape'' of the subgroups in question is an artifact of different conventions for alternating forms in the classical and representation-theoretic contexts. 

If one hopes to use the approach taken in \cite{RSY2022} for other cases, such as the Siegel congruence subgroups of level $p^n$, then a necessary component is the determination of dimensions of spaces of fixed vectors for all irreducible smooth representations of $\GSp(4, \Q_p)$ for appropriate local subgroups. In this paper we are concerned with the Siegel congruence subgroups of level $p^n$ given by $${\rm Si}(p^n)=\begin{bmatrix}
\Z_p & \Z_p & \Z_p &  \Z_p \\
\Z_p & \Z_p & \Z_p &  \Z_p \\
p^n\Z_p & p^n\Z_p & \Z_p &  \Z_p \\
p^n\Z_p & p^n\Z_p & \Z_p &  \Z_p \\
\end{bmatrix}\cap \GSp(4, \Z_p).$$ At this time, computing the associated dimensions for all representations appears overly ambitious, so we restrict ourselves to the very special case of depth zero supercuspidals (constructed in the next section). These representations arise in two ways; we will consider one of them here and the other in a companion paper. The representations of $\GSp(4, \Q_p)$ considered in this paper will all be nongeneric. While these restrictions are a limitation, some interesting phenomena already arise. Perhaps the most significant of these is a strong dichotomy between the cases $p=2$ and $p>2$. 

For applications, it is also useful to understand the effect of the Atkin-Lehner involution on the space of Siegel-invariant vectors. We compute the signature of this involution for the local representations considered.

In a previous paper \cite{Co2024}, we considered the analogous question of fixed vectors for higher depth Klingen subgroups. It was shown in that paper that the representations considered here had no fixed vectors for any $n$, nor was there a parity dichotomy. Thus, in our situation, the theory of Siegel vectors is substantially richer than that of Klingen vectors.

Here is a brief outline of this paper. Section 2 sets up notation and definitions. In section 3 we reduce the dimension computation to the determination and enumeration of a certain set of double cosets, together with a collection of character table computations for a certain finite reductive group. The main result, our dimension formula, is Theorem \ref{Odd q dimension formula}. The formula for the Atkin-Lehner signature is Theorem \ref{AL signature formula}. 

 We thank Ralf Schmidt for many useful conversations in preparing this article.

\section{Setup}

Let $F$ be a finite extension of $\Q_p$, with ring of integers $\OF$, maximal ideal $\p$, uniformizer $\varpi$, and residue field $\Fq$. We assume that if $p=2$ then $F/\Q_2$ is unramified. Let $e=\val_F(2)\in \{0, 1\}$ and $$\GL_{2,2}(q):=\left\{  (g, h)\in \GL_2(q)^2 : \det(g)=\det(h)  \right\}.$$ Let  $J = \begin{bsmallmatrix}
& & & 1\\
& & 1 & \\
 &-1 & & \\
 -1 & & & 
\end{bsmallmatrix}$, $G=\GSp(4, F)=\{ g\in \GL(4, F): {^t}gJg=\mu(g)J  \}$, $Z=Z(G)$, $T\subset G$ the set of diagonal matrices, and $G^0=\{g\in G: \mu(g)\in \OF^\times  \}$. We will need the two maximal compact subgroups $K=\GSp(4, \OF)$ and $$\K{}=\begin{bsmallmatrix}
\OF & \OF & \OF &  \p^{-1} \\
\p & \OF & \OF &  \OF \\
\p & \OF & \OF &  \OF \\
\p & \p & \p &  \OF 
\end{bsmallmatrix}\cap G^0$$ as well as the standard Iwahori subgroup $$I=\begin{bsmallmatrix}
\OF^\times & \OF & \OF &  \OF \\
\p & \OF^\times & \OF &  \OF \\
\p & \p & \OF^\times &  \OF \\
\p & \p & \p &  \OF^\times 
\end{bsmallmatrix}\cap G$$ and the sequence of Siegel congruence subgroups $$\Si{n}=\begin{bsmallmatrix}
\OF & \OF & \OF &  \OF \\
\OF & \OF & \OF &  \OF \\
\p^n & \p^n & \OF &  \OF \\
\p^n & \p^n & \OF &  \OF 
\end{bsmallmatrix}\cap G^0$$ where $n$ is a positive integer. Let $\K{}^+$ be the prounipotent radical of $\K{}$. Given $\varpi$, there is an  isomorphism $\K{}/\K{}^+\to \GL_{2,2}(q)$ defined by $$\begin{bsmallmatrix}
a & * & * & \varpi^{-1} b \\
* & a' & b' & * \\
* & c' & d' & * \\
\varpi c & * & * & d
\end{bsmallmatrix} \mapsto  \left( \begin{bsmallmatrix}
a & b \\ c & d
\end{bsmallmatrix} , \begin{bsmallmatrix}
a' & b' \\ c' & d'
\end{bsmallmatrix} \right) \pmod {\p}. $$  For representations $\rho_1$ and $\rho_2$ of $\GL_2(q)$, we write $[\rho_1\boxtimes \rho_2]$ for the restriction of $\rho_1\boxtimes \rho_2$ from $\GL_2(q)^2$ to $\GL_{2,2}(q)$. Every irreducible representation of $\GL_{2,2}(q)$ occurs as a subrepresentation of $[\rho_1\boxtimes \rho_2]$ for some $\rho_1, \ \rho_2$.  

Define the matrices $$
s_1 = \begin{bsmallmatrix}
& 1 & &  \\
1 &  &  & \\
& & & 1 \\
& & 1 & 
\end{bsmallmatrix}, \qquad
s_2 = \begin{bsmallmatrix}
1&  & &  \\
  &  & 1  & \\
&  -1 & &  \\
& &  & 1 
\end{bsmallmatrix}, \qquad u_n =  \begin{bsmallmatrix}
&  & 1 &  \\
  &  &   & -1 \\
\varpi^n &   & &  \\
& -\varpi^n &  &  
\end{bsmallmatrix}.$$  The images of $s_1$ and $s_2$ generate the Weyl group $W:=N_G(T)/T$, and $u_n\in N_G(\Si{n})$ is the Atkin-Lehner element. Note $u_1\in N_G(\K{})$ and represents the unique nontrivial coset in $N_G(\K{})/Z(G)\K{}$. 

For an irreducible complex representation $\sigma$ of a group, we write $\omega_\sigma$ for its central character.  If $H$ is a subgroup of $G$ and $\rho:H\to \GL(W)$ is a representation of $H$, then $\cInd_H^G(\rho)$ is the space of functions $f:G\to W$ such that $f(hg)=\rho(h)f(g)$ for all $h\in H$, $g\in G$, with $f$ being smooth and having compact support modulo $H$. This space is a smooth representation of $G$ under right translation. 

If $\pi$ is a depth zero supercuspidal representation of $G$ then $\pi$ arises in one of two ways. Either $\pi = \cInd_{ZK}^G(\rho)$ where $\rho|_K$ is an inflation if an irreducible cuspidal representation of $K/K^+=\GSp(4, q)$, or  $\pi = \cInd_{N_G(\K{})}^G (\tau)$ where $\tau$ is an irreducible representation of $N_G(\K{})$ and $\tau|_{\K{}}$ is an inflation of a cuspidal representation of $\K{}/\K{}^+=\GL_{2,2}(q)$. In the second case, $\tau|_{\K{}}$ is either $\sigma$ or $\sigma\oplus \sigma^{u_1}$, with the latter case occuring if and only if $\sigma \not\cong \sigma^{u_1}$; here $\sigma^{u_1}(g) := \sigma(u_1^{-1}gu_1)$. %Since $\pi$ is supercuspidal, it has no fixed vectors for an Iwahori subgroup, and in particular $\pi^{\Si{}}=0$. We therefore consider only fixed vectors for $\Si{n}$ with $n\geq 2$. 
We will compute $\dim \pi^{\Si{n}}$ for all such $\pi$ that come from $\K{}$. In a companion article we will handle the case of $K$.  We mention without proof that for the nongeneric $\pi$ arising from $K$, one has $\pi^{\Si{n}}=0$ for all $n\geq0$. So this article contains all the interesting cases of Siegel vectors in nongeneric depth zero supercuspidals.

\section{Siegel vectors for \texorpdfstring{$\pi = \cInd_{N_G(\K{})}^G (\tau)$}{}}

For $g\in G$ and fixed $n\geq 0$ define the objects \begin{eqnarray}
[g] &=& N_G(\K{})g\Si{n}\nonumber \\
R_g &=& ( g\Si{n}g^{-1} \cap \K{})\K{}^+/\K{}^+ \nonumber\\
\Supp(\pi) &=& \{ [g] : \sigma^{R_g}\neq 0  \} = \{ [g]: \exists f \in \pi^{\Si{n}}, \ f(g)\neq 0   \}.  \nonumber
\end{eqnarray}

Then $$\dim \pi^{\Si{n}} = \sum\limits_{[g]\in \Supp(\pi)} \dim \sigma^{R_g}.$$ We therefore require a determination of $\Supp(\pi)$ and of $\dim \sigma^{R_g}$ for $[g]\in \Supp(\pi)$. Observe that $$N_G(\K{})  \cap g\Si{n}g^{-1}   =  \K{}\cap g\Si{n}g^{-1}  $$ since $\K{}$ is the unique maximal compact subgroup of $N_G(\K{})$.

\subsection{Character table computations}

We first compute $\dim \sigma^R$ for those subgroups $R$ which will turn out to be the $R_g$ above. 

\begin{remark}
When $q$ is odd, the image of $Z(\K{})$ has index $2$ in $Z(\K{}/\K{}^+)$ with a representative of the complement given by $(I_2, -I_2)$. A priori, we only know that central characters of representations of $\K{}/\K{}^+$ must be trivial on the image of $Z(\K{})$ in order for the associated compact induction to have Siegel vectors. As a consequence of the next result, we will see that the relevant representations of $\K{}/\K{}^+$ are precisely the ones with trivial central character.  
\end{remark}

\begin{lemma} \label{character table computation lemma} Let $\rho_1$ and $\rho_2$ be irreducible cuspidal representations of $\GL_2(q)$ with $\omega_{\rho_1}\omega_{\rho_2}=1$. 
%Let $\sigma$ be a cuspidal representation of $\K{}/\K{}^+$ with central character $\omega$ trivial on the image $\{ (a I_2, a I_2): a\in \Fq^\times  \}$ of $\{a I_4: a\in \OF^\times\}= Z(\K{})$.

a) If $\sigma = [\rho_1\boxtimes \rho_2]$ and $R = \left\{\left( \begin{bsmallmatrix}
a & \\ & b
\end{bsmallmatrix}, \begin{bsmallmatrix}
c & \\ & abc^{-1}
\end{bsmallmatrix}  \right):  a,b,c\in \Fq^\times\right\}$ then $$\dim \sigma^{R} = \begin{cases}
1+\omega_{\rho_i}(-1) & \text{ if } q\equiv 1\pmod 2\\
1 & \text{ if } q\equiv 0\pmod 2
\end{cases}.$$

b) If $\sigma\subsetneq [\rho_1\boxtimes \rho_2]$ and $R = \left\{\left( \begin{bsmallmatrix}
a & \\ & b
\end{bsmallmatrix}, \begin{bsmallmatrix}
c & \\ & abc^{-1}
\end{bsmallmatrix}  \right):  a,b,c\in \Fq^\times\right\}$   then $q$ is odd and $$\dim \sigma^{R} =\frac{1+\omega_{\rho_i}(-1)}{2}= \begin{cases}
1 & \text{ if } q\equiv 3\pmod 4\\
0& \text{ if } q\equiv 1\pmod 4
\end{cases}.$$ % = \frac{1+\omega_{\rho_i}(-I_2)}{2}  

c) If $\sigma = [\rho_1\boxtimes \rho_2]$ and $R = \left\{ \left( \begin{bsmallmatrix}
      a & \\
     u  & a 
      \end{bsmallmatrix} , \begin{bsmallmatrix}
      a & \\
     u  & a
      \end{bsmallmatrix} \right): a\in \Fq^\times, u\in \Fq \right\}$ then $$\dim\sigma^{R}=q-1.$$

 d) If $q$ is even, $\sigma = [\rho_1\boxtimes \rho_2]$ and $R= \left\{ \left( \begin{bsmallmatrix}
       a & \\
      u+a(b^2+b)  & a 
       \end{bsmallmatrix} , \begin{bsmallmatrix}
       a & \\
      u  & a
       \end{bsmallmatrix} \right): a\in \Fq^\times, b, u\in \Fq \right\}$ then $$\dim\sigma^{R}=1.$$ 
\end{lemma}

\begin{proof}
a) Note $\omega_{\rho_1}(a)\omega_{\rho_2}(a)=1$ so $\omega_{\rho_1}(-1)=\omega_{\rho_2}(-1)$. Let $(t_1, t_2)\in R$. Since the $\rho_i$ are cuspidal representations of $\GL_2(q)$, the character table gives that $\dim(\rho_i)=q-1$ and $\trace \rho_i(t_j)=0$ if $t_j$ is nonscalar. Thus the only contributions come from elements of the form $(aI_2, aI_2)$ and, if $q$ is odd, $(aI_2, -aI_2)$. 
So if $q$ is odd then

\begin{eqnarray}
\dim \sigma^{R} &=& |R|^{-1}\sum_{(t_1, t_2)\in R} \trace \rho_1(t_1)\trace\rho_2(t_2) \nonumber \\ 
&=& (q-1)^{-3} \sum\limits_{a\in \Fq^\times} (\dim \rho_1) (\dim \rho_2) \omega_{\rho_1}(a)\omega_{\rho_2}(a)(1+\omega_{\rho_2}(-1))  \nonumber \\
&=& 1+\omega_{\rho_2}(-1). \nonumber
\end{eqnarray}   If $q$ is even then the same computation shows that $\dim \sigma^{R}=1$.

b) This case occurs if and only if $q$ is odd (see \cite{Ros2016}, Lemma A.6) and both $\rho_i$ have reducible restriction to $\SL_2(q)$. %corresponds to a character $\Lambda_i:\mathbb{F}_{q^2}^\times \to \C^\times$ with $\Lambda_1^{q-1}=\Lambda_2^{q-1} $ equals the nontrivial quadratic character of $\mathbb{F}_{q^2}^\times$. 
There then exist equidimensional irreducible representations $\sigma_\pm$ of $\SL_2(q)$ with ${\rho_1}|_{\SL_2(q)} = {\rho_2}|_{\SL_2(q)}=\sigma_+ \oplus \sigma_-$. The character table of $\SL_2(q)$ gives $\trace \sigma_\pm \left( \begin{bsmallmatrix}
t & \\ & t^{-1}
\end{bsmallmatrix} \right) =0$ if $t\neq \pm1$. % and $\trace \sigma_+(\pm I_2)=\trace \sigma_-(\pm I_2)= \frac{q-1}{2}\Lambda_0(\pm1)$. 
If $x\in \Fq^\times$ is not a square then $\rho_i\left(\begin{bsmallmatrix}
x & \\
& 1 
\end{bsmallmatrix}\right)$ interchanges $\sigma_+$ and $\sigma_-$. It follows that $$\sigma|_{\SL_2(q)^2} \in \{ \sigma_+\boxtimes \sigma_+ \oplus \sigma_- \boxtimes \sigma_-, \   \sigma_+\boxtimes \sigma_- \oplus \sigma_- \boxtimes \sigma_+\}.$$ %In either case, this implies that $\trace \sigma(I_2, -I_2) = \frac{(q-1)^2}{2}\Lambda_0(-1)$. 
Let $(t_1, t_2)\in R$. If $\det(t_1)=\det(t_2)$ is not a square, then $\sigma (t_1, t_2)$ interchanges the two summands, so $\trace \sigma(t_1, t_2)=0$. If $\det(t_1)=\det(t_2) =a^2$ then $(t_1, t_2) =(aI_2, aI_2) (t_1' ,t_2')$ with $\det(t_i')=1$, so $\trace \sigma(t_1, t_2)=  \trace \sigma(t_1', t_2')$ and is zero unless $t_1'=t_2' = \pm I_2$ or $t_1'=-t_2'=\pm I_2$, which implies the $t_i$ are scalar matrices. Clearly $\trace \sigma (aI_2, \pm aI_2) =  \dim \sigma \cdot \omega_\sigma(1, \pm 1) = \frac{(q-1)^2}{2}  \omega_{\rho_2}(\pm 1).$ Thus %= \frac{(q-1)^2}{2} \Lambda_2(\pm 1) $. 
 \begin{eqnarray}
\dim \sigma^{R} &=& |R|^{-1}\sum_{(t_1, t_2)\in R} \trace \sigma(t_1, t_2) \nonumber \\ 
&=& (q-1)^{-3} \sum\limits_{a\in \Fq^\times} \left( \trace \sigma(aI_2, aI_2) + \trace \sigma(aI_2, -aI_2) \right) \nonumber \\
&=& \frac{ 1+\omega_{\rho_2}(-1)}{2} . \nonumber
\end{eqnarray} Finally, we claim $\omega_{\rho_2}(-1)=(-1)^{\frac{q-3}{2}}$, from which the final conclusion follows. To see this, let $t\in \mathbb{F}_{q^2}^\times$ with $t^{q-1}=-1$ so $t^2\in \Fq^\times \setminus \Fq^{\times 2}$. The representation $\rho_2$ is associated to a character $\Lambda: \mathbb{F}_{q^2}^\times\to \C^\times$ with $\Lambda^{q-1}=\alpha\circ N$, where $N:\mathbb{F}_{q^2}^\times\to \Fq^\times$ is the norm map and $\alpha:\Fq^\times \to \{\pm 1\}$ is the nontrivial quadratic character. Thus $\omega_{\rho_2}(-1) = \Lambda(-1)= \Lambda(t^{q-1}) = \alpha(t^{q+1}) = \alpha(-1)\alpha(t^2) = -\alpha(-1) = (-1)^{\frac{q-3}{2}}$.

c) From the character table for $\GL_2(q)$ and cuspidality of $\rho_i$, we have $\trace \rho_i\left( \begin{bsmallmatrix}
a & \\ u & a
\end{bsmallmatrix}  \right)=-\omega_{\rho_i}(a)$ if $u\neq 0$. Thus \begin{eqnarray}
\dim \sigma^{R} &=& |R|^{-1}\sum_{(g_1, g_2)\in R} \trace \rho_1(g_1)\trace \rho_2(g_2) \nonumber \\ 
&=& \frac{1}{q(q-1)} \sum\limits_{a\in \Fq^\times, u\in \Fq} \trace \rho_1\left(\begin{bsmallmatrix}
      a & \\
     u  & a 
      \end{bsmallmatrix}\right)\trace \rho_2\left(\begin{bsmallmatrix}
            a & \\
           u  & a 
            \end{bsmallmatrix} \right) \nonumber \\
&=& \frac{1}{q} \left( (q-1)^2+ \sum\limits_{u\in \Fq^\times} (-1)^2\right)  \nonumber \\
&=& q-1 \nonumber.
\end{eqnarray} 

d) Note $|\{ b^2+b: b\in \Fq  \}|=\frac{q}{2}$. We have \begin{eqnarray}
\dim \sigma^{R} 
&=& |R|^{-1}\sum_{(g_1, g_2)\in R} \trace \rho_1(g_1)\trace \rho_2(g_2) \nonumber \\ 
&=& \frac{2}{q^2(q-1)} \sum\limits_{a\in \Fq^\times, u\in \Fq, b^2+b\in \Fq} \trace \rho_1\left(\begin{bsmallmatrix}
      a & \\
     u+a(b^2+b)  & a 
      \end{bsmallmatrix}\right)\trace \rho_2\left(\begin{bsmallmatrix}
            a & \\
           u  & a 
            \end{bsmallmatrix} \right) \nonumber \\
&=& \frac{2}{q^2}\left( (q-1)\sum\limits_{b^2+b\in \Fq} \trace \rho_1\left( \begin{bsmallmatrix}
1 & \\
b^2 +b & 1
\end{bsmallmatrix}   \right) -\sum\limits_{u\in \Fq^\times, b^2+b\in \Fq} \trace \rho_1\left(\begin{bsmallmatrix}
      1 & \\
     u+b^2+b  & 1 
      \end{bsmallmatrix}\right)  \right)  \nonumber \\
&=& \frac{2}{q^2} \left( (q-1)^2 - (q-1)\left(\frac{q}{2}-1\right) + (q-1) -   \sum\limits_{u\in \Fq^\times, b^2+b\in \Fq^\times} \trace \rho_1\left(\begin{bsmallmatrix}
      1 & \\
     u+b^2+b  & 1 
      \end{bsmallmatrix}\right) \right) \nonumber \\       
&=& \frac{2}{q^2} \left( (q-1)\left(\frac{q}{2}+1\right)+ (q-2)\left(\frac{q}{2}-1\right)  - (q-1) \left( \frac{q}{2}-1  \right)  \right) \nonumber   \\
&=& 1. \nonumber \end{eqnarray} This completes the proof.\end{proof}

Next we perform some computations which will be necessary to determine the signature of the Atkin-Lehner involution acting on $\pi^{\Si{n}}$. Let $w=\begin{bsmallmatrix}
0 & 1 \\-1 & 0
\end{bsmallmatrix}\in \GL_2(q)$.  Observe that the conjugation action of $u_1$ on $\K{}$ descends to an action on $\K{}/\K{}^+\cong \GL_{2,2}(q)$, which can be described as the composition (in either order) of swapping factors and conjugation by $\left(w, w\right)$. 
The conditions $\sigma = [\rho_1\boxtimes \rho_2]=\sigma^{u_1}$, and $\omega_\sigma=1$ are equivalent to the given conditions on $\sigma$ in the next lemma. 

\begin{lemma} \label{swap acts trivially on all fixed spaces}
Suppose that $\sigma= [\lambda\rho\boxtimes \rho]$ for some irreducible cuspidal representation $\rho$ of $\GL_2(q)$ and $\lambda:\Fq^\times\to \C^\times$ satisfying $\omega_\rho(-1)=1$,  and $(\lambda \omega_\rho)^2 = 1$. Let $s=s_\sigma$ be the involution acting on $\sigma$ by $s(v_1\otimes v_2)=v_2\otimes v_1$. 

a) If $q$ is even then $s$ acts trivially on $\sigma^R$ for each of the three subgroups $R$ in Lemma \ref{character table computation lemma}. 

b) If $q$ is even and $R = \left\{\left( \begin{bsmallmatrix}
a & \\ & b
\end{bsmallmatrix}, \begin{bsmallmatrix}
c & \\ & abc^{-1}
\end{bsmallmatrix}  \right):  a,b,c\in \Fq^\times\right\}$ then $\sigma(w,w)$ acts trivially on $\sigma^R$. 

c) Suppose $q$ is odd and $R = \left\{\left( \begin{bsmallmatrix}
a & \\ & b
\end{bsmallmatrix}, \begin{bsmallmatrix}
c & \\ & abc^{-1}
\end{bsmallmatrix}  \right):  a,b,c\in \Fq^\times\right\}$. If $\lambda\omega_\rho=1$ then $\sigma(w,w)$ and $s$ act trivially on $\sigma^R$. If $\lambda\omega_\rho\neq 1$ % (so is the nontrivial quadratic character). 
then $\sigma(w,w)$ acts by $(-1)^{\frac{q-3}{2}}$ and $s$ acts with signature 0. 
\end{lemma}

\begin{proof}
a), b). Since $q$ is even we may replace $\rho$ by $\lambda^{-1/2}\rho$, and so assume $\lambda=\omega_\rho=1$. 

Suppose first $R= \left\{\left( \begin{bsmallmatrix}
a & \\ & b
\end{bsmallmatrix}, \begin{bsmallmatrix}
c & \\ & abc^{-1}
\end{bsmallmatrix}  \right):  a,b,c\in \Fq^\times\right\}$. Let $S = \{ \begin{bsmallmatrix}a & \\ & b
\end{bsmallmatrix} : a, b\in \Fq^\times \}$. Since $\rho$ is cuspidal, $\trace\rho\begin{bsmallmatrix}a & \\ & b
\end{bsmallmatrix}=(q-1)\delta_{a,b}$. Thus $$\dim \rho^S = \frac{1}{(q-1)^2}\sum\limits_{a,b\in \Fq^\times}\trace \rho \begin{bsmallmatrix}a & \\ & b
\end{bsmallmatrix} = 1.$$ Let $v$ be a nonzero vector in $\rho^S$. Then clearly $v\otimes v\in \sigma^{R}$ and is fixed by $s$. Since $\rho(w)v\in V^S$, we have $\rho(w)v=\pm v$ (in fact $\rho(w)v=-v$ using the character table). Thus $\sigma(w, w)(v\otimes v) = \rho(w)v\otimes \rho(w)v=v\otimes v$. Since $\dim\sigma^R=1$ by Lemma \ref{character table computation lemma}, the conclusion follows. 

Next, suppose $R=\left\{ \left( \begin{bsmallmatrix}
            a & \\
           u  & a 
            \end{bsmallmatrix} , \begin{bsmallmatrix}
            a & \\
           u  & a
            \end{bsmallmatrix} \right): a\in \Fq^\times, u\in \Fq \right\}$. Fix a nontrivial character $\psi:\Fq\to \C^\times$, and for each $a\in \Fq^\times$ let  $v_a\in\rho$ be the vector, unique up to scalars, such that $\rho\left( \begin{bsmallmatrix}
      1 & \\ u & 1
      \end{bsmallmatrix}  \right)v_a = \psi(au)v_a$ for all $u\in \Fq$. The existence and uniqueness of $v_a$ is equivalent the genericity of cuspidal representations of $\GL_2(q)$ (and uniqueness of Whittaker models), which can be proved directly with the character table. So $ \sigma\left( \begin{bsmallmatrix}
            1 & \\ u & 1
            \end{bsmallmatrix}, \begin{bsmallmatrix}
            1 & \\ u & 1
            \end{bsmallmatrix}  \right) (v_a\otimes v_a) = \psi(au)v_a\otimes \psi(au)v_a = \psi(2au)(v_a\otimes v_a)=v_a\otimes v_a$ since $q$ is even. We have $\dim \sigma^R = q-1$ by Lemma \ref{character table computation lemma}, so one basis for $\sigma^R$ consists of the vectors $\{v_a\otimes v_a: \ a\in \Fq^\times\}$. The conclusion follows. 
            
            If $R=\left\{ \left( \begin{bsmallmatrix}
                        a & \\
                       u+a(b^2+b)  & a 
                        \end{bsmallmatrix} , \begin{bsmallmatrix}
                        a & \\
                       u  & a
                        \end{bsmallmatrix} \right): a\in \Fq^\times, b, u\in \Fq \right\}$, then $\sigma^R$ is a subspace of the one considered in the previous case, hence the conclusion follows.

   c) By Lemma \ref{character table computation lemma} we have $\dim \sigma^R=2$. We will compute the signature of $\sigma(w,w)$ on $\sigma^R$ and see that it is either $2$ or $ 2(-1)^{\frac{q-3}{2}}$, according to $\lambda\omega_\rho$ being trivial and nontrivial.  This signature equals \begin{eqnarray}
   & & \frac{1}{(q-1)^3} \sum\limits_{a,b,c\in \Fq^\times} \trace \sigma(w\begin{bsmallmatrix}
   a & \\ & b 
   \end{bsmallmatrix}  ,w\begin{bsmallmatrix}
      c & \\ & abc^{-1} 
      \end{bsmallmatrix}) \nonumber \\
      &=&  \frac{1}{(q-1)^3} \sum\limits_{a,b,c\in \Fq^\times} \lambda(ab) \trace \rho(\begin{bsmallmatrix}
          & b \\ -a &  
         \end{bsmallmatrix})\trace\rho(  \begin{bsmallmatrix}
             & abc^{-1} \\-c  &  
            \end{bsmallmatrix}) \nonumber \\
       &=& \frac{1}{(q-1)^3} \sum\limits_{a,b,c\in \Fq^\times} \lambda\left(ab^{-1}\right)(\lambda\omega_\rho)^2(b)\trace \rho\left(\begin{bsmallmatrix}
                 & 1 \\ -ab^{-1} &  
                \end{bsmallmatrix}\right)^2 \nonumber\\
   &=& \frac{1}{q-1} \sum\limits_{a\in \Fq^\times} \lambda(a) \trace \rho\left(\begin{bsmallmatrix}
                    & 1 \\ -a &  
                   \end{bsmallmatrix}\right)^2 \nonumber\\
   &=& \frac{4}{q-1}\sum\limits_{a\in \Fq^\times\setminus -\Fq^{\times 2}}\lambda(a)\omega_\rho(a).       \nonumber
   \end{eqnarray} In the second equality we used the fact that the two matrices are conjugate. The characteristic polynomial of $\begin{bsmallmatrix} & 1 \\ -a & 
   \end{bsmallmatrix}$ is $x^2+a$. From the character table of $\GL_2(q)$ we get no contribution if $-a$ is a square. Let $\theta:\mathbb{F}_{q^2}^\times\to \C^\times$ with $\theta\neq \theta^q$ correspond to $\rho$. Note $\omega_\rho = \theta|_{\Fq^\times}$. If $-a$ is a nonsquare then $\trace \rho\left(\begin{bsmallmatrix}
                       & 1 \\ -a &  
                      \end{bsmallmatrix}\right)^2=(-\theta(\sqrt{-a}) -\theta(-\sqrt{-a}))^2=4\theta(\sqrt{-a})^2 = 4\theta(-a)=4\omega_\rho(a)$. This explains the last equality. If $\lambda\omega_\rho=1$, the conclusion follows since there are $\frac{q-1}{2}$ terms in the summation. If $\lambda \omega_\rho$ is the nontrivial quadratic character and $-a$ is a nonsquare then $(\lambda\omega_\rho)(a)=(\lambda\omega_\rho)(-1)(\lambda\omega_\rho)(-a)=(-1)^{\frac{q-1}{2}}(-1) =(-1)^{\frac{q-3}{2}}$. The conclusion follows. 
   
   Finally, we consider the action of $s$. Since $q$ is odd, there is an explicit model of $\rho$ on the space $\{ f:\Fq^\times \to \C^\times  \}$ in which we have $(\rho\begin{bsmallmatrix}
   a &  \\ & b
   \end{bsmallmatrix} f)(x) = \omega_\rho(b)f(\frac{ax}{b})$; we will not need the formulas for the other operators. Therefore $\sigma$ can be realized on $\{ f: \Fq^\times \times \Fq^\times \to \C^\times  \}$,  $(sf)(x,y)=f(y,x)$, and \begin{eqnarray}
(\sigma\left( \begin{bsmallmatrix}
      a &  \\ & b
      \end{bsmallmatrix}, \begin{bsmallmatrix}
         ac &  \\ & bc^{-1}
         \end{bsmallmatrix} \right)f)(x,y) 
         &=& \lambda(ab) \omega_\rho(b^2c^{-1})f\left( ab^{-1}x, ab^{-1}c^2y \right) \nonumber \\
         &=&\lambda(ab^{-1})\omega_\rho(c^{-1})f\left( ab^{-1}x, ab^{-1}c^2y \right), \nonumber 
   \end{eqnarray} so $f$ is $R$-invariant if and only if $f(x,y)=\lambda(a)\omega_\rho(c^{-1})f(ax, ac^2 y)$. Let $\varepsilon\in \Fq^\times$ be a fixed nonsquare. For $\zeta\in \{1, \varepsilon\}$, define the functions $f_1, f_2\in \sigma$ by 
   \begin{eqnarray}
   f_1(x,x\zeta d^2) &=& 
   \lambda(x^{-1})\omega_\rho(d) \delta_{1, \zeta}
  \nonumber \\
   f_2(x,x\zeta d^2) &=&  \lambda(x^{-1})\omega_\rho(d) \delta_{\varepsilon,\zeta} .\nonumber
   \end{eqnarray} Here $\delta_{a,b}$ denotes the Kronecker delta function. It is easy to show that $f_1$ and $f_2$ lie in $\sigma^R$, hence are a basis. The reader may verify that $f_1(y,x)=f_1(x,y)$ and $f_2(y,x)= (\lambda\omega_\rho)(\varepsilon) f_2(x,y)$ for all $x,y\in \Fq^\times$. The conclusion follows. \end{proof}

\begin{lemma}
Suppose $\sigma\neq \sigma^{u_1}$ is an irreducible representation of $\K{}/\K{}^+$ with $\omega_\sigma=1$ and $\tau$ is an irreducible representation of $N_G(\K{})$ containing the inflation of $\sigma$. If $R$ is a subgroup of $\K{}/\K{}^+$ and $s\in u_1\K{}/\K{}^+$ normalizes $R$, then the trace of $\tau(s)$ on $\tau^R$ is zero. 
\end{lemma}

\begin{proof}
We have $\tau^R = \sigma^R \oplus (\sigma{^{u_1}})^R$, so $\tau(s)$ acts by an operator of the form $\begin{bsmallmatrix}
&* \\ *& 
\end{bsmallmatrix}$. Alternatively, one may appeal to the formula for induced characters. 
\end{proof}

%Let $H \cong \K{}/\K{}^+\rtimes \Z/2\Z$ generated by $\K{}/\K{}^+$ and $u_1$. 

%If $\sigma|_{\SL_2(q)^2} = \sigma_+\boxtimes \sigma_+ \oplus \sigma_- \boxtimes \sigma_-$ we have $\trace \sigma(I_2, -I_2) = \trace \sigma_+(I_2)\trace\sigma_+(-I_2)+\trace \sigma_-(I_2)\trace\sigma_-(-I_2)=\frac{(q-1)^2}{2} \Lambda_0(-1)$. 

%If $\sigma|_{\SL_2(q)^2} = \sigma_+\boxtimes \sigma_- \oplus \sigma_- \boxtimes \sigma_+$, we have $\trace \sigma(I_2, -I_2) = \trace \sigma_+(I_2)\trace\sigma_-(-I_2)+\trace \sigma_-(I_2)\trace\sigma_+(-I_2)=\frac{(q-1)^2}{2} \Lambda_0(-1)$

%For the remainder of this article, we assume that $\sigma$ is an irreducible cuspidal representation of $\K{}/\K{}^+$ with trivial central character. 

\subsection{Determination of \texorpdfstring{$\Supp(\pi)$}{} }

In this section we obtain a parametrization of  $\Supp(\pi)$. We have $G = I N_G(T) I$ by a well-known decomposition (see \cite{Iw1965}). Since $\K{}\supset I$ and $\K{}$ contains a set of representatives for $N_G(T)/T$, %$\supset \langle t_{-1, 1}s_1, s_2\rangle$, 
using the Iwahori factorization for $I$ we may take double coset representatives for $N_G(\K{}) \backslash  G/\Si{n}$ to have the form $t_{i, j}S(x, y,z)$ where $$t_{i, j}  : = \begin{bsmallmatrix}
\varpi^{2i+j} & & & \\
 & \varpi^{i+j} & & \\
 & & \varpi^i  & \\
 & & & 1
\end{bsmallmatrix}  ,  \qquad  S(x, y,z):=\begin{bsmallmatrix}
1 &  & & \\
 & 1 & & \\
x & y & 1 & \\
z & x &  & 1
\end{bsmallmatrix}$$ with $x, y, z\in \p$. Since $$[t_{i,j}S(x, y, z)] = [t_{-1, 1}s_1t_{i,j}S(x, y, z) s_1] = [t_{-i-1, 2i+j+1} S(x, z,y )]$$ we can and do assume $i\geq 0$.  Since \begin{eqnarray}
[t_{i,j}S(x, y, z)] &=& [S(\varpi^{-i-j}x,0,0)t_{i,j}S(0, y, z)]\nonumber \\
&=& [S(0,\varpi^{-j}y,0)t_{i,j}S(x, 0, z)]\nonumber \\
&=& [S(0,0,\varpi^{-2i-j}z)t_{i,j}S(x, y, 0)]\nonumber 
\end{eqnarray}

we may assume $\val(x)\leq i+j$ or $x=0$, $\val(y)\leq j-1$ or $y=0$, and $\val(z)\leq 2i+j$ or $z=0$, respectively. We now obtain additional relations on the valuations of $x$, $y$, and $z$.

\begin{lemma} \label{first double coset inequalities}
Suppose that $i\geq 0$ and consider the double coset $[t_{i,j}S(x,y,z)]$. 

a) If $x\neq 0$ we may assume that $0 \leq i+\val(y/x)$. 

b) If $xyz\neq 0$ and $\val(z)=2i+1+\val(y)$ then we may assume that $i+\val(y/x)\leq e-1$. 
\end{lemma}

\begin{proof} Let $A = \begin{bsmallmatrix}
1&  & & \\
c & 1 & & \\
 & & 1 &  \\
 & & -c & 1
\end{bsmallmatrix}$ where $c\in \p^{i+1}$. Then $$[t_{i,j}S(x,y,z)] = [(t_{i,j}At_{i,j}^{-1})t_{i,j}(A^{-1}S(x,y,z)A)] = [t_{i,j}S(x+cy, y, z+2cx+c^2y)].$$

a) If $\val(x)>i+\val(y)$ take $c=-x/y\in \p^{i+1}$ so $[t_{i,j}S(x,y,z)] = [t_{i,j}S(0, y, z-x^2/y)]$. 

b) Let $k=\val(zy/x^2)$. If $i+ \val(y/x)\geq  e$ then $k = 1+2(i+\val(y/x)) \geq 1+2e$ so there exists a unique $s\in 1+\p^{k-e}$ with $s^2 = 1-zy/x^2$. Taking $c=-\frac{x}{y}(1-s)\in \p^{\val(x/y)+k-e}=\p^{2i+1+\val(y/x)-e}\subset \p^{i+1}$ shows $[t_{i,j}S(x,y,z)] = [t_{i,j} S(xs, y, 0)]$. 
\end{proof}

%\begin{proof} Let $A = \begin{bmatrix}
%a& b & & \\
%c & d & & \\
% & & a & -b \\
% & & -c & d
%\end{bmatrix}$ where $\begin{bmatrix}
%a& b \\
%c & d
%\end{bmatrix}\in \GL_2(\OF)$ and $c\in \p^{i+1}$. Then $$[t_{i,j}S(x,y,z)] = [(t_{i,j}At_{i,j}^{-1})t_{i,j}(A^{-1}S(x,y,z)A)] = [t_{i,j}S(x', y', z')]$$ where \begin{eqnarray*}
%x' & =& (ad+ bc)x + cdy+abz \\
%y ' & =& d^2y + 2bdx + b^2z\\
%z' & =& a^2z + 2 acx + c^2 y.
%\end{eqnarray*} 

%a) If $\val(x)>i+\val(y)$ then taking $a=d=1$, $b=0$ and $c=-x/y\in \p^{i+1}$ shows $[t_{i,j}S(x,y,z)] = [t_{i,j}S(0, y, z')]$. 

%b) Let $e=\val(2)$ and $k=\val(zy/x^2)$. If $i+ \val(y/x)\geq  e$ then $k = 1+2(i+\val(y/x)) \geq 1+2e$ so there exists a unique $s\in 1+\p^{k-e}$ with $s^2 = 1-zy/x^2$. Taking $a=d=1$, $b=0$, and $c=-\frac{x}{y}(1-s)\in \p^{\val(x/y)+k-e}=\p^{2i+1+\val(y/x)-e}\subset \p^{i+1}$ shows $[t_{i,j}S(x,y,z)] = [t_{i,j} S(x', y, 0)]$. 
%\end{proof}

Define the two subgroups $U_1, U_2 \subset \GL_{2,2}(q)$ by  $$U_1 = \left( \begin{bsmallmatrix}
1 & \\
* & 1
\end{bsmallmatrix}, I_2\right), \qquad U_2 = \left( I_2 , \begin{bsmallmatrix}
1 & \\
* & 1
\end{bsmallmatrix}\right)$$ and observe that if $R_g\supset U_1$ or $R_g\supset U_2$ then $\tau^{R_g}=0$ by cuspidality of $\sigma$.

\begin{lemma}\label{first support inequalities}  Let $i\geq 0$. If $[t_{i,j}S(x, y,z)] \in \Supp(\pi)$ then:

a) We must have $1\leq j\leq  n-2-2i$. 

b) We may assume $\val(x^2)\geq 2i+j+2$ and $\val(y^2)\geq j+1$,

c) We may assume $\val(z) = 2i+1+\val(y)$.

\end{lemma}

\begin{proof} Let $g = t_{i,j}S(x,y,z)$.

a) If $2i+j\geq n-1$ then take $c\in \p^{2i+j+1-n}$ and compute $$g\begin{bmatrix}
1 & & & \\
& 1 & & \\
&  & 1 & \\
\varpi^n c & & & 1
\end{bmatrix}g^{-1} = \begin{bmatrix}
1 & & & \\
& 1 & & \\
&  & 1 & \\
\varpi^{n-2i-j}c & & & 1
\end{bmatrix}.$$ So $R_g \supset U_1$. If $j\leq 0$ then take $b\in \OF$ and compute $$g\begin{bmatrix}
1 & & & \\
xb & 1 & b & \\
& & 1 & \\
& & -xb & 1
\end{bmatrix}g^{-1} = \begin{bmatrix}
1 & & & \\
 & 1-yb & \varpi^jb & \\
& \varpi^{-j}y^2b & 1-yb & \\
\varpi^{-2i-j}x^2b & &  & 1
\end{bmatrix}.$$ 
Observe that $\varpi^{-j}y^2b\in \p$. If $x=0$ then taking $b\in \p^{-j}$ shows $R_g$ contains a conjugate of $U_2$. If $x\neq 0$, then $2i+j-2\val(x)+1\neq -j$ by parity. So let $M:= \max( 2i+j-2\val(x)+1, -j) \geq 0$ and take $b\in \p^M$. Thus $R_g$ contains either $U_1$ or a conjugate of $U_2$.

b) Suppose $\val(x^2)\leq 2i+j+1$, so $x\neq 0$ and we may assume $0\leq i+ \val(y/x)$ by Lemma \ref{first double coset inequalities}. Let $b\in \p^{2i+j+1-\val(x ^2)}$ and compute $$g\begin{bmatrix}
     1 & & & \\
      & 1 & b & \\
     & & 1 & \\
     & &  & 1 
     \end{bmatrix}g^{-1} = \begin{bmatrix}
     1 & & & \\
      -\varpi^{-i}xb & 1-yb & \varpi^j b & \\
      \varpi^{-i-j}xyb & \varpi^{-j}y^2b  & 1-yb & \\
      -\varpi^{-2i-j}x^2b & \varpi^{-i-j}xyb & \varpi^{-i}xb  & 1 
     \end{bmatrix}.$$ Since $j\geq 1$ and $\val(x)\leq i+j$, we have  $\varpi^{-i}xb\in \p^{i+j-\val(x)}\subset \p$, $\varpi^{-i-j}xyb\in \p^{i+1 +\val(y/x)}\subset \p$, and $\varpi^{-j}y^2b\in \p^{1+2(i+\val(y/x))}\subset \p$. So $R_g \supset U_1$ and we may assume $\val(x^2)\geq 2i+j+2$. 
     
 If $0\leq i+\val(y/x)$ then $2i+j+2\leq \val(x^2) \leq 2i+\val(y^2)$ so $\val(y^2)\geq j+2$.  So if $\val(y^2)\leq j$ then $\val(x)>i+\val(y)\geq i+1$ and we repeat the previous computation, taking now $b\in \p^{j-\val(y^2)}$ and noting then that $xb\in \p^{i+1}$, $xyb\in \p^{j+\val(x/y)}\subset \p^{i+j+1}$, and $x^2b\in \p^{j+2\val(x/y)}\subset \p^{2i+j+2}$. So $R_g\supset U_2$ and we may assume that $\val(y^2)\geq j+1$.

 c) Suppose $\val(z)\neq 2i+1+\val(y)$. If $1\leq \val(y)\leq j-1$ or $1\leq \val(z)\leq 2i+j$ then let $M :=\max(j-\val(y), 2i+j+1-\val(z))\geq1$ and take $a\in 1+\p^M$. Then compute $$g \begin{bmatrix}
        1 & & & \\
        & a & & \\
        & & 1 & \\
        & & & a
        \end{bmatrix} g^{-1} = \begin{bmatrix}
        1 & & & \\
        & a & & \\
        &   \varpi^{-j}y(a-1) & 1 &  \\
        \varpi^{-2i-j}z(1-a) & & & a
        \end{bmatrix}$$ so either $R_g \supset U_1$ or $R_g \supset U_2$. On the other hand, if $\val(y)\geq j$ and $\val(z)\geq 2i+j+1$ then $[g]=[t_{i,j} S(x,0,0)]$. 
\end{proof}

\begin{lemma}\label{reducing to the levi}
Suppose $[g]\in \Supp(\pi)$ and let $M = \begin{bsmallmatrix}
* & * & & \\
 * & * & & \\
  & & * & * \\
   & & * & * 
\end{bsmallmatrix}\cap G$ be the Siegel Levi subgroup of $G$. Then \begin{eqnarray}
[g] &=& N_G(\K{})g(\Si{n}\cap M)\nonumber \\
R_g &=& ( g(\Si{n}\cap M)g^{-1} \cap \K{})\K{}^+/\K{}^+ \nonumber.
\end{eqnarray}
\end{lemma}

\begin{proof} We may take $g= t_{i,j}S(x,y,z)$ with $i\geq 0$, $1\leq j\leq n-2-2i$, $\val(x^2)\geq 2i+j+2$, $\val(y^2)\geq j+1$, and $\val(z)=2i+1+\val(y)$.  
Since $2i+j\leq n-2$ we have $$g\begin{bmatrix}
I_2 & \\
\varpi^n C & I_2
\end{bmatrix}g^{-1} \subset \K{}^+$$ for all $C = \begin{bmatrix}
c_1 & c_2 \\
c_3 & c_1
\end{bmatrix} \in M_2(\OF)$. Since $\val(xy), \val(yz)\geq i+j+1$ and  $\val(xz), \val(z^2)\geq 2i+j+2$ we have 
$$g\begin{bmatrix}
I_2 & B \\
 & I_2
\end{bmatrix}g^{-1} \subset \K{}^+$$ for all $B = \begin{bmatrix}
b_1 & b_2 \\
b_3 & b_1
\end{bmatrix} \in M_2(\OF)$. Every $s\in \Si{n}$ can be written (uniquely) as $$s=\begin{bmatrix}
I_2 & \\
\varpi^n C & I_2
\end{bmatrix} \begin{bmatrix}
A & \\
& \lambda A'
\end{bmatrix} \begin{bmatrix}
I_2 & B \\
& I_2
\end{bmatrix}$$ with the factors in any order. The conclusions all follow. 
\end{proof}

Thanks to Lemma \ref{first support inequalities}, if $[g]\in \Supp(\pi)$ then we may assume that $g$ takes on one of four forms: $g=t_{i,j}$, $g=t_{i,j}S(x,0,0)$, $g=t_{i,j}S(0, y, z)$, or $g=t_{i,j}S(x, y, z)$ where $\val(z)=2i+1+\val(y)$. In the last case, when $xyz\neq 0$, Lemma \ref{first double coset inequalities} then shows that we may assume $q$ is even and $i+\val(y/x)=0$. We now complete the determination of necessary and sufficient conditions to have $[g]\in \Supp(\pi)$ for each of these forms.

 \begin{lemma} \label{keyLemma}  Suppose $i\geq 0$, $1\leq j\leq n-2-2i$, $\val(x^2)\geq 2i+j+2$, $\val(y^2)\geq j+1$, and $\val(z)=2i+1+\val(y)$.

a) If $q$ is even, or if $q$ is odd and $\sigma$ has trivial central character, then $[t_{i,j}]\in \Supp(\pi)$ and $$R_{t_{i,j}} = \{\left( \begin{bsmallmatrix}
a & \\ & b
\end{bsmallmatrix}, \begin{bsmallmatrix}
c & \\ & abc^{-1}
\end{bsmallmatrix}  \right)  : a,b,c\in \Fq^\times\}.$$

  b) If $\val(x)\leq i+ j$ then $[t_{i,j}S(x,0,0)]\in \Supp(\pi)$ if and only if $q$ is even and $\val(x)=i+j$, in which case 
  $$R_{t_{i,j}S(x,0,0)} = \{\left( \begin{bsmallmatrix}
  a & \\ & b
  \end{bsmallmatrix}, \begin{bsmallmatrix}
  c & \\ & abc^{-1}
  \end{bsmallmatrix}  \right)  : a,b,c\in \Fq^\times\}.$$

c) If $\val(y)\leq j-1$ then $[t_{i,j}S(0,y,z)]\in \Supp(\pi)$ if and only $q$ is even and if $ j-\val(y)\leq 2$. % and either $j-\val(y)$ is odd or   $j-\val(y)=2e$. 

If $j-\val(y)=1$ then $R_{t_{i,j}S(0,y,z)}$ is conjugate to $$\{\left( \begin{bsmallmatrix}
      a & \\ u  & a  \end{bsmallmatrix} , \begin{bsmallmatrix}
      a & \\ u  & a \end{bsmallmatrix} \right): a\in \Fq^\times, u\in \Fq \}.$$
      
      If $j-\val(y)=2$ then $R_{t_{i,j}S(0,y,z)}$ is conjugate to $$\{\left(\begin{bsmallmatrix}
            a & \\
           u+a(b^2+b)  & a 
            \end{bsmallmatrix} , \begin{bsmallmatrix}
            a & \\
           u  & a
            \end{bsmallmatrix}\right): a\in \Fq^\times, u, b\in \Fq\}.$$

 d) Suppose $q$ is even, $\val(y)\leq j-1$, and $i+\val(y/x)=0$. Then $[t_{i,j}S(x,y,z)]\in \Supp(\pi)$ if and only if $j - \val(y) = 1$, in which case $R_{t_{i,j}S(x,y,z)}$ is conjugate to $$\{\left(\begin{bsmallmatrix}
              a & \\
             u+a(b^2+b)  & a 
              \end{bsmallmatrix} , \begin{bsmallmatrix}
              a & \\
             u  & a
              \end{bsmallmatrix}\right): a\in \Fq^\times, u, b\in \Fq\}.$$

% If these conditions hold with $M$ odd then $R_{t_{i,j}S(x,y,z)}$ is conjugate to $$\{\left( \begin{bsmallmatrix}
%       a & \\ u  & a  \end{bsmallmatrix} , \begin{bsmallmatrix}
%       a & \\ u  & a \end{bsmallmatrix} \right): a\in \Fq^\times, u\in \Fq \}.$$

 \end{lemma}
 
 We refer to the types of double cosets that appear in this lemma as types I, II, IIIa, IIIb, and IV, respectively; see Table \ref{tab:table1}. 
 
 \begin{proof}
 Let $g=t_{i,j}S(x,y,z)$. Lemma \ref{character table computation lemma} shows that if $R_g$ has any of the claimed forms and $\omega_\sigma=1$ then $[g]\in \Supp(\pi)$. Let $A = \begin{bsmallmatrix}
 a_1 & a_2 \\ a_3 & a_4
 \end{bsmallmatrix}\in \GL_2(\OF)$, and $A' = \begin{bsmallmatrix}
  a_1 & -a_2 \\ -a_3 & a_4
  \end{bsmallmatrix}$. 

a) This is clear.

  b) Suppose $\val(2x)\leq i+j$. Let $a\in \p^{i+j+1-\val(2x)}\subset \p$ and compute 
  $$g\begin{bmatrix}
  1 & & & \\
  \varpi^ia & 1 & & \\
  & & 1 & \\
  & & -\varpi^i a & 1 
  \end{bmatrix}g^{-1} = \begin{bmatrix}
  1 & & & \\
  a & 1 & & \\
   & & 1 & \\
   2x\varpi^{-i-j}a & & -a & 1 
  \end{bmatrix}.$$ Thus $R_g \supset U_1$. Conversely, suppose that $\val(2x) \geq i+j+1$ and compute 
$$g\begin{bmatrix}
   A & \\
   & \lambda A'
   \end{bmatrix}g^{-1} = \begin{bmatrix}
            a_1 & \varpi^ia_2 & & \\
           \varpi^{-i} a_3 & a_4 & & \\
         \varpi^{-i-j}xa_1(1-\lambda)   & \varpi^{-j}xa_2(1+\lambda) & \lambda a_1 & -\varpi^{i}\lambda a_2 \\
         \varpi^{-2i-j}xa_3(1+\lambda)   & \varpi^{-i-j}xa_4(1-\lambda) & -\varpi^{-i}\lambda a_3 & \lambda a_4
            \end{bmatrix}.$$. 
            
For this to lie in $\K{}$ we must have $a_3\in \p^{i+1} \subset \p$, so $a_1, a_4\in \OF^\times$, and hence $\lambda \in 1+\p^{i+j+1-\val(x)}$. Therefore $1+\lambda \in 2+\p^{i+j+1-\val(x)} \subset \p^{i+j+1-\val(x)}$, so $xa_3(1+\lambda) \in \p^{2i+j+2}$ and $xa_2(1+\lambda)\in \p^{j+1}$. Thus $R_g$ is as claimed.

 c)  Suppose $2e <j-\val(y)$. Set $a = 1+\frac{v\varpi^j}{2y}$ where $v\in \OF$. Then compute 
  $$g\begin{bmatrix}
  1 & & & \\
  & a & & \\
  & & a^{-1} & \\
  & & & 1
  \end{bmatrix}g^{-1} = \begin{bmatrix}
  1 & & & \\
  & a & & \\
  & \varpi^{-j}y(a-a^{-1}) & a^{-1} & \\
  & & & 1
  \end{bmatrix}$$ and observe that $$\varpi^{-j}y(a-a^{-1}) = v a^{-1} \left(1+\frac{v\varpi^j}{4y}\right)  \equiv v \pmod {\p} .$$ Thus $R_g \supset U_2$.

%  Next suppose that $2e > j-\val(y) \geq 1$ (so $q$ is even) and $j-\val(y)$ is even. Repeat the above computation taking now $a = 1+v\varpi^l$ where $l=\frac{1}{2}(j-\val(y))<e$ and $v\in \OF$. Then $$\varpi^{-j}y(a-a^{-1}) =va^{-1} \left(\frac{y}{\varpi^{\val(y)}} \right)(v+2\varpi^{-l})\equiv v^2 \left( \frac{y}{\varpi^{\val(y)}}  \right)\pmod {\p}. $$ Varying $v$ (using that $q$ is even) shows $R_g \supset U_2$.
%   

  % Conversely, suppose  $ j-\val(y)\leq 2e$, and either $j-\val(y)$ is odd or $j-\val(y)=2e$. 
                     
  We may now assume $q$ is even so $e=1$ and $2\geq j-\val(y)\geq 1$. Then
   \begin{equation} \label{matrix 1}
g\begin{bmatrix}
      A & \\
      & \lambda A'
      \end{bmatrix}g^{-1} = \begin{bmatrix}
      a_1 & \varpi^ia_2 & & \\
      \varpi^{-i}a_3 & a_4 & & \\
     \varpi^{-i-j}(ya_3 +z\lambda a_2) & \varpi^{-j}y(a_4 - \lambda a_1) & \lambda a_1 & -\varpi^i\lambda a_2 \\
    \varpi^{-2i-j}z(a_1 - \lambda a_4)  & \varpi^{-i-j}(za_2 + y\lambda a_3 ) & -\varpi^{-i}\lambda a_3 & \lambda a_4
      \end{bmatrix}. 
   \end{equation} It is clear that if this matrix lies in $\K{}$, then its image in $\K{}/\K{}^+$ is independent of $a_2$ and $a_3$. So for the purposes of determining $R_g$, we may assume that $a_2=a_3=0$. Let $a_4 = \lambda a_1 + u\varpi^jy^{-1}$ with $u\in \OF$. Then $a_1-\lambda a_4 = a_1(1-\lambda^2) -\lambda u\varpi^jy^{-1}$. Since $a_1\in \OF^\times$, it is necessary and sufficient that $$\val(\lambda^2-1) \geq j-\val(y)$$ in order for the matrix to lie in $\K{}$; assume this holds. In particular, $\lambda\pm 1 \in\p$ since $q$ is even. So $\val(\lambda^2-1) = \val(\lambda-1)+\val(\lambda+1) \geq 2$. %if $j-\val(y)=1$. If $1<\val(\lambda-1)$ then $1=\val(\lambda-1+2) = \val(\lambda+1)$ so $\val(\lambda^2-1) = \val(\lambda-1)+\val(\lambda+1) >2 \geq j-\val(y)$. %If $1=\val(2)>\val(\lambda-1)$ then $\val(\lambda+1)=\val(\lambda-1)$  and $2>\val(\lambda^2-1) = 2\val(\lambda-1)$. This forces $j-\val(y)\neq \val(\lambda^2-1)$ since $j-\val(y)$ is odd if less than $2$, hence $\val(\lambda^2-1)>j-\val(y)$ in this case as well. 
      %Finally, if $1=\val(\lambda-1)$ then  $\val(\lambda^2-1) \geq 2 \geq j-\val(y)$. So if $\val(\lambda^2-1)=j-\val(y)$ then $j-\val(y)=2$ as claimed.  

       If $\val(\lambda^2-1) >j-\val(y)$ then the image of the matrix (\ref{matrix 1}) above in $R_g$ will be $$\left( \begin{bsmallmatrix}
      a_1 & \\
     s u  & a_1 
      \end{bsmallmatrix} , \begin{bsmallmatrix}
      a_1 & \\
     u  & a_1
      \end{bsmallmatrix} \right) \pmod{\p}$$ where $s := zy^{-1}\varpi^{-2i-1} \in \OF^\times$. Conjugating by $\left( \begin{bsmallmatrix}
                  t & \\
                  & t^{-1}
                  \end{bsmallmatrix} , \begin{bsmallmatrix}
                              1 & \\
                              & 1
                              \end{bsmallmatrix} \right)$ where $t^2 \equiv s \pmod{\p}$, we obtain the claimed form of $R_g$ when $j-\val(y)=1$.

      If $\val(\lambda^2-1)=j-\val(y)=2$ then  $\val(\lambda-1)=\val(2)=1$ so $\lambda^2-1 = 4v(v+1)$ for some $v\in \OF^\times$ and the corresponding element of $R_g$ is $$\left( \begin{bsmallmatrix}
      a_1 & \\
      su + s'a_1v(v+1)  &  a_1
      \end{bsmallmatrix} , \begin{bsmallmatrix}
            a_1 & \\
           u   &  a_1
            \end{bsmallmatrix}  \right) \pmod {\p}$$ where $s':=4z\varpi^{-2i-j-1} \in \OF^\times$. Conjugating by $\left( \begin{bsmallmatrix}
            t & \\
            & t^{-1}
            \end{bsmallmatrix} , \begin{bsmallmatrix}
                        r & \\
                        & r^{-1}
                        \end{bsmallmatrix} \right)$ where $t^2\equiv  s'\pmod{\p} $ and $r^2\equiv s' s^{-1} \pmod{\p}$, we obtain the claimed form of $R_g$.

   d) Let $t=\varpi^iyx^{-1}\in \OF^\times$. For $\alpha \in F$ we have $$g\begin{bmatrix}
        1 & & & \\
        \alpha x/y & 1+\alpha & & \\
        & & 1+\alpha & \\
        & & -(1+\alpha)\alpha x/y & (1+\alpha)^2 \end{bmatrix}g^{-1} = \begin{bmatrix}
             1 & & & \\
             \alpha/t & 1+\alpha & & \\
            0 & 0 & 1+\alpha & \\
            m & 0 & -(1+\alpha)\alpha/t& (1+\alpha)^2 \end{bmatrix}$$ where $$m = \frac{\alpha x^2}{\varpi^{2i+j}y} (2+\alpha)\left(1-\frac{zy}{x^2}\right).$$

 Let $M:= j + \val(z/x^{2})  =2(i+\val(y/x)+1)+j-\val(y)-1\geq 2$. %and $M = j + \val(zx^{-2})$. 
    Suppose first that $M>2e$. Let $r\in \OF$ and take $\alpha =  \frac{ry\varpi^{2i+j+1}}{2x^2} \in \p^{M-e} \subset \p^{e+1}$. We have $\alpha/t \in \p^{e-\val(y/x)-i+1} \subset \p$ and $$m = \left( \varpi  r  \right) \left(1+\frac{\alpha}{2}\right) \left(1 - \frac{zy}{x^2} \right) \equiv \varpi r\pmod{\p^2}.$$ Thus $R_g\supset U_1$. 
    
  We may now assume $q$ is even so $e=1$ and $M=2$, so $j-\val(y)=1$. Then 
  \begin{equation} \label{matrix 2}
g\begin{bmatrix}
   A & \\
   & \lambda A'
   \end{bmatrix}g^{-1} = \begin{bmatrix}
   a_1 & \varpi^i a_2 & & \\
   \varpi^{-i}a_3 & a_4 & & \\
\varpi^{-i-j}m_1  & \varpi^{-j}m_2 &    \lambda a_1 & -\varpi^i\lambda a_2  \\
\varpi^{-2i-j}m_3  & \varpi^{-i-j}m_4 &    -\varpi^{-i}\lambda a_3 & \lambda a_4 
   \end{bmatrix} 
  \end{equation}  
   where \begin{eqnarray}
m_1&=& xa_1(1-\lambda)+ya_3 + z\lambda a_2  \nonumber\\ 
  m_2&=& xa_2(1+\lambda)+y(a_4-\lambda a_1)   \nonumber\\
   m_3  &=& xa_3(1+\lambda)+z(a_1-\lambda a_4)   \nonumber\\
     m_4  &=& xa_4(1-\lambda)+\lambda ya_3 + za_2.    \nonumber
   \end{eqnarray}        Assuming the matrix (\ref{matrix 2}) lies in $\K{}$, its image in $\K{}/\K{}^+ = \GL_{2,2}(q)$ is $$\left(  \begin{bsmallmatrix}
   a_1 & 0 \\
   \varpi^{-2i-j-1}m_3   & \lambda a_4
 \end{bsmallmatrix} , \begin{bsmallmatrix}
 a_4 & 0 \\
 \varpi^{-j}m_2 & \lambda a_1
  \end{bsmallmatrix} \right).$$ It is necessary that $a_3\in \p^{i+1}$ and $a_1, a_4, \OF^\times$ in order for the matrix (\ref{matrix 2}) to lie in $\K{}$. Let $u = zy^{-1}\varpi^{-2i-1} \in \OF^\times$, $a_3' = \varpi^{-i-1}a_3\in \OF$ and $a_2' = \varpi^{i}a_2\in \p^{i} $.  The remaining conditions to lie in $\K{}$ can be written as 
  
%  \begin{eqnarray}
%  a_1(1-\lambda) + ta_3' + tu\lambda a_2'&\in & \p^{i+j+1-\val(x)}  \nonumber\\ 
%    a_2'(1+\lambda) + \varpi t(a_4-\lambda a_1)&\in & \p^{i+j+1-\val(x)}  \nonumber\\ 
%     a_3'(1+\lambda) + \varpi t u(a_1-\lambda a_4) &\in & \p^{i+j+1-\val(x)}  \nonumber\\
%     a_4(1-\lambda) + \lambda t a_3' + tua_2'&\in & \p^{i+j+1-\val(x)} .  \nonumber 
%     \end{eqnarray} 
%  
%  
%  Since $0=\val(t) = i + \val(y/x)<i+j+1-\val(x)$, we must have $\lambda \pm 1\in  \p$. We thus rewrite our conditions as 
  
  \begin{eqnarray}
  S_1: = \frac{m_1}{\varpi xt}  =  a_1\left(\frac{1-\lambda}{\varpi t}\right) + a_3' + u\lambda a_2'&\in & \p  \nonumber\\ 
  R_1: = \frac{ m_2}{y} =    a_2'\left(\frac{1+\lambda}{ t}\right) +  a_4-\lambda a_1 &\in & \p \nonumber\\ 
 R_2:= \frac{m_3}{\varpi^{i+1}xt}=      a_3'\left(\frac{1+\lambda}{ t}\right) + u(a_1-\lambda a_4) &\in & \p  \nonumber \\
  S_2:= \frac{m_4}{\varpi xt} =    a_4\left(\frac{1-\lambda}{\varpi t}\right) + \lambda  a_3' + ua_2'&\in & \p. \nonumber
       \end{eqnarray}  In particular, we must have $\lambda-1\in \p$ so $\val(\lambda^2-1) \geq 2$. We write $$a_3' = a_1\left(\frac{\lambda-1}{ \varpi t}\right) - u\lambda a_2' + S_1 \in \OF$$  
       $$a_4 = \lambda a_1 - a_2' \left(\frac{1+\lambda}{ t}\right) + R_1 \in \OF^\times, $$ noting this implies $a_1\equiv a_4\pmod{\p}$, and plug these into the other two expressions to obtain the equalities

   $$S_2 = a_2'  \left(\frac{\lambda^2-1 }{\varpi t^2}\right) ( 1 - u\varpi t^2) + R_1 \left( \frac{1-\lambda}{\varpi t}  \right) + \lambda S_1$$
   
   $$R_2=a_1\left( \frac{\lambda^2-1}{\varpi t^2}  \right)(1-u\varpi t^2) - \lambda u R_1 + \left(\frac{1+\lambda}{t} \right)S_1$$ which yields

   $$\lambda R_2=\left( \frac{\lambda^2-1}{\varpi t^2}  \right)(1-u\varpi t^2) \left( a_1\lambda - a_2' \left(\frac{1+\lambda}{t}\right) + R_1  \right) - u R_1 + \left(\frac{\lambda+1}{t}\right)S_2.$$ Note that $ \left(\frac{\lambda+1}{t}\right)S_2\in \p^2$, so $$\frac{\lambda R_2}{\varpi} \equiv \left(\frac{\lambda^2-1}{\varpi^2t^2}\right)a_1 - \frac{uR_1}{\varpi} \pmod{\p}$$ which implies a congruence relation between $\varpi^{-2i-j-1}m_3$ and $\varpi^{-j}m_2$. The rest of the argument is similar to the previous case.    \end{proof}
   
%  Since $R_1, S_2, 1+\lambda\in \p$, we have $$\val\left( \varpi a_1\lambda - a_2' \left(\frac{1+\lambda}{t}\right) + R_1  \right)=1$$ and therefore must have %$$\val\left( \frac{\lambda^2-1}{\varpi t^2}  \right) \geq 1$$ or equivalently, 
%  $$\val(\lambda^2-1) \geq 2.$$ The rest of the argument now follows in a similar manner to the previous case. %If $\val(\lambda^2-1)>M$, then $R_2 + u\lambda R_1 \in \p^{j+1-\val(y)}$, implying the claimed form of $R_g$. By a similar argument as in the previous case, this must occur unless $M=2e$, and the claimed form of $R_g$ then also follows in a similar manner to the previous case. 

 %  If $\val(\lambda-1)>e$ then $\val(\lambda+1)=e$ and  $\val(\lambda^2-1) =\val(\lambda-1)+\val(\lambda+1)>2e \geq M$. If $ \val(\lambda-1)<e$ then $\val(\lambda+1) = \val(\lambda-1)$ and $M\leq \val(\lambda^2-1) = 2\val(\lambda-1) <2e$, hence $M$ is odd by assumption and so $M<\val(\lambda^2-1)$. Finally, suppose $\val(\lambda-1)=e$, so $\lambda-1 = 2v$ for some $v\in \OF^\times$ and $\val(\lambda^2 -1) \geq 2e \geq M$. So if $\val(\lambda^2-1)=M$ then $M = 2e$. The claimed form of $R_g$ follows in a similar manner to the previous case. 

\section{Counting double cosets}

We must enumerate the elements of $\Supp(\pi)$. Since $[g]$ determines $R_g$ up to $\K{}/\K{}^+$-conjugacy, it follows from Lemma \ref{reducing to the levi},  Lemma \ref{keyLemma}, and some tedious matrix computations that there are no equalities between double cosets of different types, and that $[t_{i,j}S(x,y,z)]\in \Supp(\pi)$ determines $i\geq 0$ and $j\geq 1$, as well as $\val(x)$ and $\val(y)$ in those cases when $x\neq0$ and/or $y\neq 0$. Since $[t_{i,j}S(x,0,0)]=[t_{i,j}S(\varpi^{\val(x)},0,0)]$ if $x\neq0$ by using integral diagonal matrices, these considerations quickly yield the following two results.

% Thanks to Lemma \ref{reducing to the levi}, we ha ve $[g]= [g']\in \Supp(\pi)$ if and only if $g' (\Si{n}\cap M)g^{-1}\cap N_G(\K{})$ is nonempty. Using this, it is not hard to verify that $[t_{i,j}]$ determines $i$ and $j$ when $i\geq 0$ and $j\geq1$. Hence we have the following result. 

%We make the following observations. First, the double coset $$N_G(\K{}) t_{i,j}K = \left(\K{} t_{i,j}ZK \right)\bigcup \left(\K{} u_1 t_{i,j} ZK\right)$$ determines $i$ and $j$ if $i\geq 0$ and $j\geq 1$, as can be verified by direct computation. It follows that $[t_{i,j}S(x,y,z)]$ determines $i$ and $j$. Next, it is not hard to show that none of the different types of double cosets can be equal (for example, because the associated subgroups $R_g$ are not conjugate), and that $[t_{i,j}S(x,y,z)]$ also determines $\val(x)$ and $\val(y)$, for the types where $x\neq 0$ and/or $y\neq 0$. Using integral diagonal matrices, we may assume that $x$ and $z$ are zero or powers of $\varpi$. The only ambiguity that remains is counting the number of distinct double cosets of the form $[t_{i, j}S(0, u\varpi^k, \varpi^{2i+1+k})]$ and of $[t_{i,j} S(\varpi^r, u\varpi^k  ,\varpi^{2i+1+k})]$, where $u\in \OF^\times / (1+\p^{j-k})$. 

%THIS WORKS FOR ALL $n\geq 0$
\begin{lemma}\label{diagonal double coset count}
The number of double cosets of type I is $$\left\lfloor  \frac{(n-1)^2}{4}\right\rfloor. $$
\end{lemma}

\begin{proof}
 Compute $\sum\limits_{j=1}^{n-2}\sum\limits_{i=0}^{ \left\lfloor \frac{n-j-2}{2} \right\rfloor  } (1) = \sum\limits_{j=1}^{n-2}  \left\lfloor \frac{n-j}{2} \right\rfloor= \sum\limits_{j=2}^{n-1} \left\lfloor \frac{j}{2}\right\rfloor  =  \left\lfloor  \frac{(n-1)^2}{4} \right\rfloor .$ 
\end{proof}

%Since the conjugacy class of $R_g$ is an invariant of $[g]$, it follows from Lemma \ref{keyLemma} that if $[t_{i,j}S(x,y,z)]=[t_{i',j'}]\in \Supp(\pi)$ then we may take $y=z=0$. It is then straightforward to show by direct matrix computation that double cosets of type II are distinct from those of type I, and that for type II the double coset $[t_{i,j}S(x,0,0)] =[t_{i,j}S(\varpi^{\val(x)},0,0)] $ (using diagonal matrices) determines $i, j$ and $\val(x)$. Hence we obtain the following count for those of type II. 

%THIS WORKS FOR ALL $n\geq 0$
\begin{lemma}\label{type II count} If $q$ is even then the number of double cosets of type II % the form $[t_{i,j}S(\varpi^)]\in \Supp(\pi)$ 
is $$\begin{cases}
0 & \text{ if } n\leq 3 \\
 \left\lfloor \frac{(n-2)^2}{4}\right\rfloor& \text{ if } n\geq 4 
\end{cases}.$$
 
\end{lemma}

\begin{proof} Since $e=1$ we must have $i+j=\val(x)\geq i+1+\frac{j}{2}$. So the relevant inequalities are $2\leq j\leq  n-2-2i$ and $0\leq i\leq \left\lfloor \frac{n-4}{2} \right\rfloor$.  Compute $  \sum\limits_{i=0}^{\left\lfloor \frac{n-4}{2} \right\rfloor }(n-3-2i) =  \left\lfloor \frac{(n-2)^2}{4}\right\rfloor $. 
\end{proof}
  
Suppose now that $g = t_{i,j}S(0,y,z)$ is type III (so $q$ is even). Using integral diagonal matrices, we may assume that $z$ is a power of $\varpi$. So the only ambiguity that remains is counting the number of distinct double cosets of the form $[t_{i, j}S(0, u\varpi^r, \varpi^{2i+1+r})]$, where $u\in \OF^\times / (1+\p^{j-r})$ and $i, j, r$ are fixed. Using integral diagonal matrices again, it is easy to see that we may adjust $u$ by squares. A tedious matrix computation shows that there are no additional equalities, so that $(i, j, r, u)$ form a complete set of invariants if $u\in \OF^\times /\OF^{\times2} (1+\p^{j-r})$. In particular, there are $q^{\lfloor \frac{j-r}{2} \rfloor }$ such double cosets for fixed $i, j, r$.

\begin{lemma} \label{type IIIa count}  
If $q$ is even and $n\geq 2$ then the number of double cosets of type IIIa is   $$\begin{cases}
0 & \text{ if } n\leq 3 \\
\left\lfloor \frac{(n-3)^2}{4} \right\rfloor.& \text{ if } n\geq 4 
\end{cases}.$$

\end{lemma}

%The number of double cosets of the form $[t_{i,j} S(0, \varpi^{j+1-2k}, u\varpi^{2i+2+j-2k} ) ]$ is as follows. 
\begin{proof} Since $e=1$ and $j-r$ is odd we must have $j-1=r\geq \frac{j+1}{2}$ so the relevant inequalities are $3\leq j\leq n-2-2i$ and $0\leq i\leq \left\lfloor \frac{n-5}{2}\right\rfloor $.  Compute
$ \sum\limits_{i=0}^{ \lfloor \frac{n-5}{2} \rfloor  } (n-4-2i)
%&=& \sum\limits_{k=1}^{M} q^{k-1}\left\lfloor \frac{(n+1-4k)^2}{4} \right\rfloor\nonumber \\\nonumber\\
 =    \left\lfloor \frac{(n-3)^2}{4} \right\rfloor$. 
\end{proof}

%The number is $$\sum\limits_{j=2}^{n-2} \sum\limits_{i=0}^{ \left\lfloor \frac{n-j-2}{2} \right\rfloor  } \sum\limits_{k=1}^{\min(e, \left\lfloor j/2\right\rfloor)} (1) = \sum\limits_{j=2}^{n-2}\left\lfloor \frac{n-j}{2} \right\rfloor \min(e, \left\lfloor j/2\right\rfloor).$$ If $2e+3 \geq n$ then $e\geq \left\lfloor j/2 \right\rfloor $ and the sum equals $\sum\limits_{j=2}^{n-2}\left\lfloor \frac{n-j}{2} \right\rfloor \left\lfloor \frac{j}{2}\right\rfloor = \frac{n-1}{6}\left\lfloor \frac{n(n-2)}{4} \right\rfloor $. If $2e+3\leq  n$ then the sum becomes 
%\begin{eqnarray}
%\sum\limits_{j=2}^{2e-1}\lfloor  \frac{n-j}{2} \rfloor \lfloor \frac{j}{2}\rfloor + \sum\limits_{j=2e}^{n-2}\lfloor \frac{n-j}{2} \rfloor e & = &\sum\limits_{k=1}^{e-1}\left( \lfloor \frac{n}{2} \rfloor  + \lfloor \frac{n-1}{2} \rfloor -2k \right)k  + e\sum\limits_{l=2}^{n-2e}\lfloor \frac{l}{2}\rfloor  \nonumber  \\
%& = &  \frac{(n-1)(e-1)e}{2} - \frac{(e-1)e(2e-1)}{3} + e \left\lfloor \frac{(n-2e)^2}{4} \right\rfloor \nonumber \\
%& =& \frac{e(e-1)(3n-4e-1)}{6} + e \left\lfloor \frac{(n-2e)^2}{4} \right\rfloor \nonumber 
%\end{eqnarray} 

\begin{lemma} \label{type IIIb count}
If $q$ is even then the number of double cosets of type IIIb is $$\begin{cases}
0 & \text{ if } n\leq 3 \\
q \left\lfloor \frac{(n-5)^2}{4} \right\rfloor& \text{ if } n\geq 4 
\end{cases}.$$ %\begin{cases}
%0 & \text{ if } n\leq  4e+2  \\
%q^e\cdot \left\lfloor \frac{(n-4e-1)^2}{4} %\right\rfloor  & \text{ if }  n\geq  4e+3
%\end{cases}$$
\end{lemma}

\begin{proof}
The conditions $j=2+r\leq \min(n-2, 2r-1)$ can only be satisfied if $n\geq 7$. If $n\geq 7$ then the number of double cosets is $\sum\limits_{j=5}^{n-2} \sum\limits_{i=0}^{ \lfloor \frac{n-j-2}{2} \rfloor  } | \OF^\times/\OF^{\times 2}(1+\p^{2})  |  =  q \sum\limits_{j=5}^{n-2} \left\lfloor \frac{n-j}{2} \right\rfloor 
 =  q  \left\lfloor \frac{(n-5)^2}{4} \right\rfloor$.\end{proof}
%see Ralf note, page 8

Finally, suppose $g=t_{i,j} S(x, y ,z)$ is type IV (so $q$ is even). Using integral diagonal matrices, we may assume that $x$ and $z$ are powers of $\varpi$. So the only ambiguity that remains is counting the number of distinct double cosets of the form $[t_{i,j} S(\varpi^k, u\varpi^r  ,\varpi^{2i+1+r})]$, where $u\in \OF^\times / (1+\p^{j-r})$ and $i, j, k, r$ are fixed. Observe that the conditions in Lemma \ref{keyLemma} part d) ensure that $k=i+j-1$ and $r=j-1$, so in particular we have $u\in \OF^\times/(1+\p)$. A tedious matrix computation shows that if $[t_{i,j} S(\varpi^{i+j-1}, u\varpi^{j-1}  ,\varpi^{2i+j})] = [t_{i,j} S(\varpi^{i+j-1}, u'\varpi^{j-1}  ,\varpi^{2i+j})] $ then $u-u'\in \p$. Thus there are $q-1$ double cosets for fixed $i, j$.

\begin{lemma} \label{type IVb count}
If $q$ is even then the number of double cosets of type IV is $$\begin{cases}
0 & \text{ if } n\leq 3\\
(q-1) \left\lfloor \frac{(n-4)^2}{4}\right\rfloor& \text{ if } n\geq 4
\end{cases}.$$ 
\end{lemma}

\begin{proof} From the above considerations, and since $i+j-1=\val(x)\geq i+1+\frac{j}{2}$, the relevant inequalities are $4\leq j\leq n-2-2i$ and $0\leq i\leq \lfloor \frac{n-6}{2} \rfloor$. Since there are $q-1$ double cosets for each $i,j$ our count is 
$(q-1) \sum\limits_{i=0}^{\lfloor  \frac{n-6}{2}\rfloor } (n-5-2i)
%& =&  \sum\limits_{k=0}^{N-1}  \left(\left\lfloor \frac{m}{2} \right\rfloor -k\right) \left(\left\lceil \frac{m}{2}\right\rceil -k\right)q^k \nonumber \\
=  (q-1)    \left\lfloor \frac{(n-4)^2}{4}\right\rfloor$.\end{proof}

%$$\dim \pi^{\Si{n}} =  \left\lfloor \left(\frac{n-1}{2}\right)^2 \right\rfloor 
%+\left\lfloor \left(\frac{n-2}{2}\right)^2 \right\rfloor 
%+ \left\lfloor \left(\frac{n-4}{2}\right)^2 \right\rfloor  
%+ 2\left\lfloor \left(\frac{n-5}{2}\right)^2 \right\rfloor 
%+ \frac{1}{2}\left\lfloor \left(\frac{n-3}{2}\right) \right\rfloor \left\lfloor \left(\frac{n-1}{2}\right) \right\rfloor   $$

% old formula, think is wrong now %$$\dim \pi^{\Si{n}} = 1+(n-4)^2+ \left\lfloor \frac{(n-1)^2}{4} \right\rfloor
%+ \frac{1}{2}\left\lfloor \left(\frac{n-3}{2}\right) \right\rfloor \left\lfloor \left(\frac{n-1}{2}\right) \right\rfloor   $$

\section{Main dimension formulas}

If $\tau|_{\K{}} = \sigma\oplus \sigma^{u_1}$ %(which is equivalent to $\sigma \neq \sigma^{u_1}$ and to $\tau$ being induced from $\sigma$), which is the necessary and sufficient condition for $\pi$ to lie in an $L$-packet with a generic supercuspidal representation, namely one of the form $X_5(\Lambda, \omega)$ or $\chi_4(k, l)$, 
 then $\dim \tau^{R_g} =  \dim\sigma^{R_g} + \dim(\sigma^{u_1})^{R_g}$. If $\tau|_{\K{}} = \sigma $ (which is equivalent to $\sigma = \sigma^{u_1}$), then either $\sigma = [\rho\boxtimes \lambda_0\rho]$, in which case $\dim \tau^{R_g} = \dim \sigma^{R_g}$, or (when $q$ is odd) $\sigma \subsetneq [\rho\boxtimes \lambda_0\rho]$, in which case $\dim \tau^{R_g} = \dim \sigma^{R_g} $ is 1 or 0 if $q\equiv 3\pmod 4$ or $q\equiv 1\pmod 4$, respectively.  We know the contribution of the diagonal double cosets for these nongeneric representations, and these are the only double cosets when $q$ is odd, so we have the following result. 
 
 \begin{theorem}\label{Odd q dimension formula}
Let $\pi = \cInd_{N_G(\K{})}^G(\tau) $ be an irreducible depth zero supercuspidal representation with trivial central character. Let $\sigma$ be an irreducible constituent of $\tau|_{\K{}}$ and of $[\rho_1\boxtimes \rho_2]$.  

\begin{itemize}
\item  Suppose $q$ is odd. If $\omega_\sigma\neq 1$ then $\dim \pi^{\Si{n}}=0$. If $\omega_\sigma=1$ then $$\dim \pi^{\Si{n}} =\left\lfloor \frac{(n-1)^2}{4} \right\rfloor \cdot \begin{cases}
4 & \text{ if }  \sigma \neq \sigma^{u_1}\\
2 & \text{ if }  \sigma = \sigma^{u_1}=[\rho_1\boxtimes \rho_2]\\
1 & \text{ if }  \sigma \subsetneq [\rho_1\boxtimes \rho_2]
 \end{cases}.$$ 
 
 \item Suppose $q$ is even. For $n\geq 0$ define $$F(n,q):=\begin{cases}
  0 & \text{ if } n\leq 2\\
 1 & \text{ if } n=3 \\
 2n-5 + q(\left\lfloor\frac{3n^2+1}{4}\right\rfloor-6n+12) & \text{ if } n\geq 4
  \end{cases}.$$Then $$\dim\pi^{\Si{n}}= F(n, q) \cdot \begin{cases}
 2 & \text{ if }  \sigma \neq \sigma^{u_1}\\
 1 & \text{ if }  \sigma = \sigma^{u_1}
  \end{cases}.$$ 
\end{itemize}
 %The third case can occur if and only if $q\equiv 3\pmod 4$. %, the first if and only if $q\geq 5$. I THINK THIS PART WAS WRONG - JUST TAKE $\rho\boxtimes \rho_0$ where $\rho$ is the unique cuspidal irrep of $GL_2(3)$ which is irreducible when restricted to $SL_2(3)$ and $\rho_0$ is the one which is reducible upon that restriction
 \end{theorem}
 
 \begin{proof}
 a) This follows from the preceding remarks, the character table computations above, together with Lemmas \ref{keyLemma} and \ref{diagonal double coset count}.

b)  This follows from the preceeding remarks, the character table computations above, Lemmas \ref{diagonal double coset count}, \ref{type II count}, \ref{type IIIa count}, \ref{type IIIb count}, %\ref{type IVa count}, 
\ref{type IVb count}, and some routine algebra. 
% \begin{eqnarray}
%F(n, q, e) &=& \left\lfloor  \frac{(n-1)^2}{4}  \right\rfloor + m(n,e) + m_2(n,q, e) + m_4(n,q,e) \nonumber \\ 
%&+&   (q-1)(m_1(n,q,e)+m_3(n,q, e) )\nonumber.
% \end{eqnarray} Plugging the expressions in gives the formula after some routine algebra. 
\end{proof}

\begin{corollary}
Suppose $\pi = \cInd_{N_{G}(\K{}) }^G (\tau)$ is a depth zero (nongeneric) supercuspidal representation of $\GSp(4, \Q_2)$ with trivial central character; there are two such $\pi$, both arising from the unique cuspidal representation of $\GL_{2,2}(2)$. We have $$\dim \pi^{\Si{n}} = \begin{cases}
0 & \text{ if } n\leq 2\\
1 & \text{ if } n=3\\
\left\lfloor \frac{3n^2-20n + 39}{2} \right\rfloor  & \text{ if } n\geq 4
\end{cases}.$$ 
\end{corollary}

\section{The Atkin-Lehner involution}
In Table \ref{tab:table1} and in this section, we make the following abbreviations: \begin{eqnarray}
X_{k}&=& S(\varpi^{k}, 0,0),\nonumber \\
Y_{i,j,r}(u) &=& t_{i,j}S(0,\varpi^r u, \varpi^{2i+1+r}), \nonumber \\
Z_{i,j}(u) &=& t_{i,j}S(\varpi^{i+j-1},\varpi^{j-1} u, \varpi^{2i+j})\nonumber.
\end{eqnarray} 

The goal of this section is to compute the signature of the Atkin-Lehner involution acting on $\pi^{\Si{n}}$. We assume below that $\omega_\sigma=1$ so that we can have $f\in \pi^{\Si{n}}\neq 0$. %Note that some situations below will require $n$ to be odd while others will require $n$ to be even. 
 
 \begin{lemma}\label{first AL lemma}
 Suppose $f\in \pi^{\Si{n}}$, and let $u_n\in N_G(\Si{n})$ denote the Atkin-Lehner element. Then the following equalities hold:
 
 a) If $g=t_{i,j}$ represents a type I double coset then $$(\pi(u_n)f)(t_{i,j})=\tau(u_1)f(t_{i,n-2i-j} ).$$ In particular, $[g u_n]=[g]$ if and only if $n$ is odd and $n=2i+2j+1$. If $n$ is odd then there are $\frac{n-1}{2}$ double cosets of type I fixed by $u_n$. 
 
 b) If $g=t_{i,j}X_k$ represents a type II double coset (so $q$ is even and $k=i+j$) then $$(\pi(u_n)f)(t_{i,j}X_k)=f(t_{i,j+n-2k}X_{n-k}).$$ In particular, $[g]=[gu_n]$ if and only if $n$ is even and $n=2k$. If $n\geq 4$ is even then there are $\frac{(n-2)}{2}$ double cosets of type II fixed by $u_n$. 
 
 c) If $g=Y_{i,j,r}(u)$ represents a type III double coset (so $q$ is even and $1\leq j-r\leq 2$) then $$(\pi(u_n)f)(Y_{i,j,r}(u))=\tau(u_1)\tau\left( \begin{bsmallmatrix}
 & u \\ 1 & 
 \end{bsmallmatrix}, \begin{bsmallmatrix}
 & 1 \\ u  & 
 \end{bsmallmatrix}  \right)f(Y_{i,j+n-2r-2i-1, n-r-2i-1}(u)).$$ In particular, $[g]=[gu_n]$ if and only if $n$ is odd and $n=2i+2r+1$. If $n\geq 3$ is odd there are $\frac{ n-3 }{2}$ double cosets of type IIIa fixed by $u_n$, and $\begin{cases}
 0 & \text{ if } n\leq 3 \\
 q\left(\frac{n-5}{2}\right) & \text{ if } n\geq 5 
 \end{cases}$ double cosets of type IIIb fixed by $u_n$. 
  % \begin{cases}
  %0 & \text{ if } n\leq  4e+1  \\
 % q^e\left( \frac{n-1-4e}{2} \right)  & \text{ if }  n\geq  4e+3
  %\end{cases}$$ 
 
 %$$q^M\left( \frac{ n+1-4M }{2(q-1)}+\frac{2}{(q-1)^2}  \right)-\left( \frac{ n+1 }{2(q-1)}+\frac{2}{(q-1)^2}  \right)$$ 

  d) If $g=Z_{i,j}(u)$ represents a type IV double coset (so $q$ is even) then $$(\pi(u_n)f)(Z_{i,j}(u))=f(Z_{i,n-2i-j+2}(u)).$$ In particular, $[g]=[gu_n]$ if and only if $n$ is even and $n=2i+2j-2$. If $n\geq 4$ is even there are $   (q-1)(\frac{n-4}{2})$ double cosets of type IV fixed by $u_n$.

 \end{lemma}
 
 \begin{proof}
 a) The first claims follow from the identity $$t_{i,j}u_n = \varpi^{i+j}u_1t_{i, n-2i-j-1} .$$
 %\in N_G(\K{})t_{i, n-2i-j-1}$$
Assuming $n$ is odd, one then has the $\frac{n-1}{2}$ double cosets $[t_{i, \frac{n-1}{2}-i}]$, with $0\leq i\leq \frac{n-3}{2}$.

b) The first claims follow from the identity $$t_{i,j}X_k u_n = \varpi^k\begin{bsmallmatrix}
-1&  & \varpi^{i+j-k} & \\
& 1 & &  -\varpi^{i+j-k}\\
& & 1 & \\
& & & -1
\end{bsmallmatrix}t_{i,j+n-2k}X_{n-k}.$$ Assuming $n$ is even, one then has the fixed double cosets $[t_{i,j}X_{i+j}]$ where $j=n/2-i\geq 2$ (since $i+j\geq i+1+\frac{j}{2}$ by Lemma \ref{keyLemma}) so $0\leq i\leq \frac{n-4}{2}$ and the count is $\frac{(n-2)}{2}$ as claimed. 

%we have $[t_{i,j}X_k u_n] = [t_{i, j+n-2k} X_{n-k}]$, and $[t_{i,j}X_ku_n]=[t_{i,j}X_k]$ if and only if $n=2k$, in which case $$f(t_{i,j}X_ku_n) = \tau\left( \begin{bmatrix}
%-1 & & & \\
%& 1 & & \\
% %& & & -1
%\end{bmatrix}  \right)f(t_{i,j}X_k)=f(t_{i,j}X_k)$$ since $q$ is even. 
%Recall $\dim\tau^{R_{t_{i,j}X_k }}$ being 2 or 1, depending on $\sigma\neq \sigma^{u_1}$ and $\sigma=\sigma^{u_1}$, respectively. 

c) Let $h = \begin{bsmallmatrix}
 &   &  &  \varpi^{-1}\\
&  & 1 &  \\
& 1 & -\frac{\varpi^{2i+j+1}}{z} & \\
\varpi & & & -\frac{\varpi^j}{y}
\end{bsmallmatrix}  \begin{bsmallmatrix}
\frac{\varpi^{n+j-1}}{y^2} & & & \\
& \frac{\varpi^{n+i+j}}{yz} & & \\
& & \frac{z}{y\varpi^{i+1}} & \\
 & & & 1
\end{bsmallmatrix}$ and observe the identity $$t_{i,j}S(0,y,z)u_n = \varpi^{i}y u_1h S\left(0, \frac{\varpi^n }{z},\frac{\varpi^n }{y}\right)\begin{bsmallmatrix}
1 & & & \\
 & 1 & & \\
& & -1 & \\
 & &  & -1
\end{bsmallmatrix}.$$ Next note that if $y=u\varpi^r$ and $z=\varpi^{2i+1+r}$, with $u\in \OF^\times$, then  $$S\left(0, \frac{\varpi^n}{z}, \frac{\varpi^n}{y}\right) = \begin{bsmallmatrix}
1 & & & \\
& 1  & & \\
 & & u^{-1} & \\
 & & & u^{-1}
\end{bsmallmatrix}S(0, u\varpi^{n-2i-1-r}, \varpi^{n-r})\begin{bsmallmatrix}
1 & & & \\
& 1 & & \\
 & & u & \\
 & & & u
\end{bsmallmatrix}.$$ Noting that $h$ introduces a factor of $\begin{bsmallmatrix}
u^{-2} & & & \\
& u^{-1} & & \\
& & u^{-1} & \\
& & & 1
\end{bsmallmatrix}$, the first two claims follow. 
%Thus $$f(g u_n)  = \tau(u_1\begin{bmatrix}
%& & & r\varpi^{-1}\\
%& & 1 & \\
%& r & & \\
%\varpi & & & 
%\end{bmatrix})  f(g)= \tau(u_1)\tau( \begin{bmatrix}
%& r \\1 &  
%\end{bmatrix}, \begin{bmatrix}
%& 1 \\ r &  
%\end{bmatrix}  )f(g).$$ 

%In order to have $[gu_n]=[g]$ we must have $\val(yz)=n$, and since $\val(z)=2i+1+\val(y)$, this is the same as $\val(y) = \frac{n-1-2i}{2}$. Thus $n$ must be odd. Write $z=\varpi^{2i+1+k}=\varpi^{n-k}$ and $y=r \varpi^k$ for some $r\in \OF^\times$.  So we see that the condition on $\val(y)$ is also sufficient to ensure $[gu_n]=[g]$. 

To count the number of double cosets of type IIIa fixed by $u_n$, note that $r=j-1$ and $i=(n+1)/2-j$. The reader may verify that the relevant inequalities are $3\leq j\leq \frac{n+1}{2}$. Hence the count is $\frac{n-3}{2}$. To count the number of double cosets of type IIIb fixed by $u_n$, since $r=j-2$ and $i=(n+3)/2-j$, the relevant inequalities are $5\leq j\leq\frac{n+3}{2}$. There are $q$ double cosets for each $j$, as seen in the proof of Lemma \ref{type IIIb count}. Hence the count is as claimed.

%If $2k=n-1-2i$, then we have $1\leq j-k = i+j - \frac{n-1}{2}\leq 2e$, so $$\frac{n+1}{2}\leq i+j\leq \frac{n-1+4e}{2}.$$ If, further, $j-k=2e$ then $n-2\geq 2i+j = i + 2e + \frac{n-1}{2}$ so $0\leq i\leq \frac{n-3-4e}{2}$. We conclude that there are $q^e\left( \frac{n-1-4e}{2}\right)$ double cosets if $n\geq 4e+3$, while there are zero double cosets if $n\leq 4e+1$ (note $n$ is odd). If instead $j-k<2e$ then $j-k = 2m+1$ for some $0\leq m \leq e-1$. We have $n-2\geq 2i+j = i + 2m + \frac{n+1}{2}$ so $0\leq i\leq \frac{n-5}{2}-2m$, hence $m\leq \lfloor \frac{n-5}{4}\rfloor$. Letting $N = \min(e-1, \lfloor \frac{n-5}{4}\rfloor)\geq 0$, the number of double cosets is $\sum\limits_{m=0}^{N} \sum\limits_{i=0}^{\frac{n-5}{2}-2m}q^{m} =\sum\limits_{m=0}^{N} (\frac{n-3}{2}-2m )q^{m}= \frac{n-3}{2}\cdot \frac{q^{N+1}-1}{q-1}-2q\frac{Nq^{N+1}-Nq^N - q^N + 1}{(q-1)^2}=:d(n)=:d(n,q,e).$ Note if $n=1$, $n=3$, or $e=0$ then $N=-1$ and this expression is zero, as it should be. Note also that $N = e-1$ if and only if $n\geq 4e+1$, or $\frac{n-1}{2}\geq 2e$. 

d)   Let $\lambda = 1-\frac{yz}{x^2}$, 
$h = \begin{bsmallmatrix}
\frac{\varpi^{2i+j+n}}{x^2} & & & \\
& \frac{\varpi^{i+j+n}}{x^2\lambda} & & \\
& &\lambda \varpi^i & \\
 & & & 1
\end{bsmallmatrix}$, 
$B=\begin{bsmallmatrix}
1 & & & \\
\frac{-z}{\lambda\varpi^ix} & -1 & & \\
& & \lambda^{-1} & \\
& & \frac{ z}{\lambda^2\varpi^i x} & -\lambda^{-1}
\end{bsmallmatrix}$ and observe the identity $$t_{i,j}S(x,y,z)u_n = x\begin{bsmallmatrix}
-1 &  \frac{\varpi^i y}{x} & \frac{\varpi^{i+j}}{x\lambda}& \\
& -1 & \frac{-\varpi^j z}{x^2\lambda} & \frac{\varpi^{i+j}}{x} \\
& & 1 & \frac{\varpi^i y}{x}\\
& & & 1
\end{bsmallmatrix} Bh S\left(\frac{\varpi^n}{x}, \frac{\varpi^n y}{x^2},\frac{\varpi^n z}{x^2}\right)\begin{bsmallmatrix}
1 & & & \\
 & 1 & & \\
& & \lambda & \\
 & &  & \lambda
\end{bsmallmatrix}.$$ Noting that $\val(x)=i+\val(y) =\val(z)-i-1$ and $\val(\frac{yz}{x^2}) =  1$, the identity shows the first two claims. 

To count the number of double cosets of type IV fixed by $u_n$, one argues in a manner similar to the previous cases. 
 \end{proof}

 \begin{theorem}\label{AL signature formula}
Let $\pi = \cInd_{N_G(\K{})}^G(\tau) $ be an irreducible depth zero supercuspidal representation with trivial central character. Let $\sigma$ be an irreducible constituent of $\tau|_{\K{}}$ and assume $\omega_\sigma=1$. Let $s(\pi, n)$ denote the signature of the Atkin-Lehner involution on $\pi^{\Si{n}}$. 

a) Suppose $n\geq 4$ is even. If $q$ is odd then $s(\pi, n)=0$. If $q$ is even then $$s(\pi, n) =\left(1+\frac{q(n-4)}{2} \right)\cdot \begin{cases}
 2 & \text{ if } \sigma\neq \sigma^{u_1} \\
 1 & \text{ if } \sigma=\sigma^{u_1} 
\end{cases}$$
     
   %  $$F(n,q,e)=\begin{cases}
     %    \frac{n(n-2)}{8}+ (q-1) (q^{\left\lfloor \frac{n-4}{4} \right\rfloor }a_q(n-4\left\lfloor \frac{n-4}{4} \right\rfloor)-a_q(n))  & \text{ if } n\leq  2e+2  \\
     %    q^{\left\lfloor \frac{n}{4}\right\rfloor}\left( n+1-4\left\lfloor \frac{n}{4}\right\rfloor+\frac{4}{q-1}  \right)  -\frac{4}{q-1}-n-1+(q-2)\frac{e(e-n+1)}{2} & \text{ if }  2e+4\leq n\leq  4e
     %     \\
      %         q^{e}\left( n-4e+1+\frac{4}{q-1}  \right) -n-1-\frac{4}{q-1}+(q-2)\frac{e(e-n+1)}{2}& \text{ if }  4e+2\leq  n
      %    \end{cases}$$ 

b) Suppose $n\geq 3$ is odd. If $q$ is even then $$s(\pi, n) = \pm \begin{cases}
0 & \text{ if } \sigma\neq \sigma^{u_1}\\
1 & \text{ if } n=3 \text{ and } \sigma = \sigma^{u_1}\\
1+q(n-4) & \text{ if } n\geq 5\text{ and } \sigma = \sigma^{u_1}
\end{cases}$$ while if $q$ is odd then $$s(\pi, n) =\pm \frac{(n-1)}{2}\begin{cases} 0 & \text{ if } \sigma\neq \sigma^{u_1} \\
2 & \text{ if } \sigma=\lambda [\rho\boxtimes \rho] , \ \lambda \omega_\rho= 1 \\
0 & \text{ if } \sigma=\lambda [\rho\boxtimes \rho] , \ \lambda  \omega_\rho=\lambda_0 \\
1 & \text{ if } \sigma\subsetneq \lambda [\rho\boxtimes \rho] 
  \end{cases}.$$  
Here the $\pm$ sign is determined by the choice of extension $\tau(u_1)$ acting on $\sigma=\sigma^{u_1}$.  %by either $\pm\text{swap}$ composed with $(w, w)$. 

\end{theorem}

\begin{proof}
a) The results of Lemma \ref{first AL lemma} show that if $[gu_n]=[g]$ then $f(gu_n)=f(g)$ since $g$ is of type II or IV. In particular, $q$ is even. Note both of these types have $\dim\sigma^{R_g}=1$ by lemma \ref{character table computation lemma}. The number of fixed double cosets was shown in Lemma \ref{first AL lemma} to be $\frac{n-2}{2}+(q-1)(\frac{n-4}{2}) = 1+\frac{q(n-4)}{2}$. The conclusion follows. 

b) The results of Lemma \ref{first AL lemma} show that if $[gu_n]=[g]$ then $g$ is of type I or III and $f(gu_n)=\tau(s)f(g)$ for some element $s$ of order $2$, with $s\in u_1\K{}$. Hence the contribution of the double coset $[g]=[gu_n]$ is equal to the trace of $\tau(s)$ on $\tau^{R_g}$. If $\sigma\neq \sigma^{u_1}$, the formula for the induced character shows that the trace of $\tau(s)$ on $\tau^{R_g} = \sigma^{R_g}\oplus \sigma{^*}^{R_g}$ is zero. Suppose now that $\sigma= \sigma^{u_1}$, so $\tau|_{\K{}}=\sigma$, and there are two choies of $\tau$ for a given $\sigma$, distinguished by $\tau(u_1)$. 

If $q$ is even then we may combine Lemma \ref{swap acts trivially on all fixed spaces} parts a) and b), together with the formulas in Lemma \ref{first AL lemma} parts a) and c), to conclude that $\tau(s)$ acts trivially on $\sigma^{R_g}$. Hence the signature is the number of fixed double cosets, weighted by $\dim \sigma^{R_g}$, which is thus $1$ if $n=3$ and, if $n\geq 5$, is  $\frac{n-1}{2}\left(1\right)+\frac{n-3}{2}(q-1)+q\left(\frac{n-5}{2}\right)(1) = 1+q(n-4)$. 

Finally, suppose $q$ is odd. In this case $s=u_1$. If $\sigma =\lambda [\rho\boxtimes \rho]$ then $\dim \sigma^{R_g}=2$ and we may apply Lemma \ref{swap acts trivially on all fixed spaces} part c), since in this case $s=u_1$ and one choice of $\tau(u_1)$ acts on $\sigma^{R_g}$ by a composition of the swapping map $s$ and $\sigma(w,w)$ where $w=\begin{bsmallmatrix} & 1 \\ -1 & 
\end{bsmallmatrix}$. If $\sigma \subsetneq \lambda [\rho\boxtimes \rho]$, then $\dim \sigma^{R_g}=1$, so $\tau(u_1)f(g)=\pm f(g)$ is automatic. The conclusion then follows from Lemma \ref{first AL lemma} part a). 
\end{proof}

%We showed above that, in all cases, $\sigma^{R_g}$ is fixed by the swap automorphism. This handles all double cosets except those of type I, on which the trace of $\tau(u_1)$ on $\tau^{R_g}=\sigma^{R_g}$ is that of $\tau(w,w)$, where $w=\begin{bsmallmatrix} & 1 \\ -1 & 
%\end{bsmallmatrix}$. We showed this action is trivial unless $q$ is odd, $\sigma = \lambda_0 \omega_\rho^{-1}[\rho\boxtimes \rho]$, in which case its trace is $2(-1)^{\frac{q-3}{2}}$. The rest now follows. 

\begin{corollary}
Suppose $F = \Q_2$, $\sigma:\GL(2, \mathbb{F}_2)^2\to \C^\times$ is the unique irreducible cuspidal representation and $\tau_\pm$ denotes the extensions to $N_G(\K{})$ by having $u_1$ act by $\pm1$ and $Z(G)$ acting trivially. Let $\pi_\pm = \cInd_{N_G(\K{})}^G(\tau_\pm)$. %Then $s(\pi_\pm, 2n+1) $ is equal to $\pm1$ times the number of double cosets preserved by $u_n$, since $\dim \sigma^{R_g}=1$ for all $g\in \Supp(\pi)$, while $s(\pi, 2n)$ is equal to the the number of double cosets preserved by $u_n$, for the same reason. 
Then $$s(\pi_\pm, 3)=\pm 1$$ and, for $n\geq 4$, we have $$s(\pi_\pm, n) =  \begin{cases} n-3  & \text{ if } n\equiv 0 \pmod 2 \\
 \pm \left( 2n-7  \right) & \text{ if } n\equiv 1\pmod 2
\end{cases}.$$
\end{corollary}

%\begin{landscape}
 \tabcolsep=0.08cm
\begin{table}[h!]
  \begin{center}
    \caption{Siegel fixed vector information, $q$ even, $n\geq 4$. If $n\leq 3$, only the first row applies. }
    \label{tab:table1}
    \begin{tabular}{c|c|c|c|c} % <-- Alignments: 1st column left, 2nd middle and 3rd right, with vertical lines in between
   $g$ & 
      $R_g$ & $\dim \sigma^{R_g}$  &   $|\{ g \}|$ &  AL\\
      \hline
      
 I  \    $t_{i,j}$  & $\left(\begin{bsmallmatrix}
      a & \\ & b 
       \end{bsmallmatrix}, \begin{bsmallmatrix}
             c & \\ & abc^{-1} 
              \end{bsmallmatrix}\right)$ & $1$ &  $\left\lfloor \frac{(n-1)^2}{4} \right\rfloor $  &   $t_{i, n-2i-j-1}$   \\

II  \   $t_{i,j}X_{i+j}$ & $\left(\begin{bsmallmatrix}
      a & \\ & b 
       \end{bsmallmatrix}, \begin{bsmallmatrix}
             c & \\ & abc^{-1} 
              \end{bsmallmatrix}\right)$  & $1$    & $\left\lfloor \frac{(n-2)^2}{4}\right\rfloor$ &  $t_{i, n-2i-j}X_{n-i-j} $  
                \\

IIIa  \  $Y_{i,j,j-1}(u)$  & $\left( \begin{bsmallmatrix}
       a & \\
      u  & a 
       \end{bsmallmatrix} , \begin{bsmallmatrix}
       a & \\
      u  & a
       \end{bsmallmatrix} \right) $ & $q-1$    &  $\left\lfloor \frac{(n-3)^2}{4} \right\rfloor $   &  $Y_{i,n-2i-j+1, n-2i-j}(u)$ 
       \\

IIIb \   $Y_{i,j,j-2}(u)$   & $ \left( \begin{bsmallmatrix}
              a & \\
             u+a(b^2+b)  & a 
              \end{bsmallmatrix} , \begin{bsmallmatrix}
              a & \\
             u  & a
              \end{bsmallmatrix} \right) $ & 1    & $q\left\lfloor \frac{(n-5)^2}{4} \right\rfloor $  & $Y_{i,n-2i-j+3, n-2i-j+1}(u)$ 
              \\

IV  \  $Z_{i,j}(u)$    &  $ \left( \begin{bsmallmatrix}
        a & \\
       u +a(b^2+b) & a 
        \end{bsmallmatrix} , \begin{bsmallmatrix}
        a & \\
       u  & a
        \end{bsmallmatrix} \right) $ & $1$ & $(q-1)\left\lfloor \frac{(n-4)^2}{4} \right\rfloor $    & $Z_{i,n-2i-j+2}(u)$
    \end{tabular}
  \end{center}
\end{table}

\clearpage

%\end{landscape}

\end{document}